\documentclass[10pt,leqno]{article}

\usepackage{amsmath,amssymb,amsthm,mathrsfs,dsfont}

\usepackage[margin=3cm]{geometry} 

\usepackage{titlesec,hyperref}

\usepackage{color}

\usepackage{fancyhdr}
\usepackage{subfigure}
\usepackage[graphicx]{realboxes}
\usepackage{float}
\usepackage{enumerate}

\titleformat{\subsection}{\it}{\thesubsection.\enspace}{1pt}{}

\newtheorem{theo}{Theorem}[section]
\newtheorem{lemm}[theo]{Lemma}
\newtheorem{defi}[theo]{Definition}

\newtheorem{prop}[theo]{Proposition}
\newtheorem{rema}[theo]{Remark}

\allowdisplaybreaks 

\begin{document}
\title{On the Critical One Components Regularity for the $3-D$ Navier-Stokes System in $L^p_T(\dot{B}^{\frac 1 2+\frac 2 p}_{2,\infty})$ spaces
	\hspace{-4mm}
}

\author{
	Wei $\mbox{Luo}^1$ \footnote{E-mail:  luow23@mail.sysu.edu.cn}, \quad
	Zhaoyang $\mbox{Yin}^{1,2}$\footnote{E-mail: mcsyzy@mail.sysu.edu.cn} \quad and
	Pei $\mbox{Zheng}^{1}$\footnote{E-mail: zhengp25@mail2.sysu.edu.cn}\\
	$^1\mbox{Department}$ of Mathematics,
	Sun Yat-sen University, Guangzhou 510275, China\\
	$^2\mbox{School}$ of Science,\\ Shenzhen Campus of Sun Yat-sen University, Shenzhen 518107, China}

\date{}
\maketitle
\hrule

\begin{abstract}
	We consider the conditional regularity of the mild solution $v$ of the $3-D$ incompressible Navier-Stokes equations with initial data $v_0\in \dot{H}^{\frac 1 2}$ and vorticity $\Omega_0\in L^{r_0}$ for some $r_0\in (1,2)$. We prove that if the solution associated with initial data $v_0$ blows up at a finite time $T^\ast$, then for any $2<p<\infty$, and any unit vectors $e$ in $\mathbb{R}^3$, the integral
	$$\int_0^{T^\ast}\left\Vert (v(t)|e)_{\mathbb{R}^3}\right\Vert_{\dot{B}^{\frac 1 2+\frac 2 p}_{2,\infty}}^p{\rm d}t$$
blows up at $T^\ast$. The conclusion improves the recent results in Chemin et al. (Arch Ration Mech Anal 224(3):871-905, 2017) and Han et al. (Arch. Rational Mech. Anal. 231:939-970, 2019).

\vspace*{5pt}
				
				\vspace*{5pt}
\end{abstract}

\vspace*{10pt}

\tableofcontents

\section{Introduction}

\quad In this article, we consider the Cauchy problem of the three dimensional incompressible Navier-Stokes equation
\begin{equation}\label{NS}
	\left\{
	\begin{aligned}
		&\partial_t v+v\cdot\nabla v-\Delta v+\nabla P=0,\ \ &&x\in\mathbb{R}^3,\ t>0,\\
		&{\rm div}\  v=0, &&x\in\mathbb{R}^3,\ t>0,\\
		&v(0,x)=v_0(x), &&x\in\mathbb{R}^3.
	\end{aligned}
	\right.
\end{equation}
Here $v:[0,\infty)\times \mathbb{R}^3\rightarrow\mathbb{R}^3$ represents the velocity field of the fluid flow and $P:[0,\infty)\times \mathbb{R}^3\rightarrow\mathbb{R}$ denotes the pressure. The equations \eqref{NS} remains invariant under the following natural scaling transformation: if $(v,P)$ is a solution to \eqref{NS}, then for any $R>0$, the rescaled functions
$$v_R(t,x)=Rv(R^2t,Rx),\ P_R(t,x)=R^2P(R^2t,Rx)$$
also form a solution, corresponding to the rescaled initial data $v_{0,R}(x)=Rv_0(Rx)$. This scaling symmetry plays a fundamental role in the well-posedness theory of the Navier-Stokes equations.
The existence of global weak solutions to \eqref{NS} is known since the famous
work of Leray \cite{Leray1934} (see also Hopf \cite{Hopf1951} for the bounded domain case) for initial data
$v_0\in L^2(\mathbb{R}^3)$ with ${\rm div}\ v_0 = 0$. The uniqueness and global regularity of Leray-Hopf weak solutions is still one of the most challenging open problems. For instance, the Prodi-Serrin-Ladyzhenskaya criterion claims that if
$$\int_0^{T^\ast}\Vert v(t,\cdot)\Vert_{L^p}^q{\rm d}t<\infty,\  \frac 2 q+\frac 3 p=1$$
for some $3<p\leq\infty$, then $v$ is still regular at time $T^\ast<\infty$, based on the works of Prodi \cite{Prodi1959}, Serrin \cite{Serrin1963}, and Ladyzhenskaya \cite{Lady1967}. In particular, Escauriaza, Ser\"{e}gin and \v{S}ver\'{a}k \cite{Iskauriaza2003} proved that if
$$\limsup\limits_{t\uparrow T^\ast}\int_{\mathbb{R}^3}|v(t,x)|^3{\rm d}x<\infty,$$
then $v$ is still regular at time $T^\ast<\infty$. Kozono and Shimada \cite{Kozono2000} proved the blow up criterion
$$\int_0^{T^\ast}\Vert v\Vert_{BMO}^2{\rm d}t=+\infty.$$
And Chen and Zhang \cite{Chen2006} proved the blow up criterion in the endpoint Besov space,
$$\int_0^{T^\ast}\Vert v\Vert_{\dot{B}^\sigma_{\infty,\infty}}^p{\rm d}t=+\infty$$
with $\frac 2 p=1+\sigma,1<p<\infty,-1<\sigma<1$.

In recent years, many works \cite{2,5,1,4,3} have established  the blow-up criterion by some of the components' information.
Cao and Titi \cite{Cao2008} proved that the blow up criterion in supercritical space
$$
\left\{
\begin{aligned}
	&v^3\in L^q([0,T^\ast],L^p(\Omega))\ {\rm with }\ p>\frac 7 2,q\in[1,\infty[,\ {\rm and}\ \frac 3 p+\frac 2q<\frac{2(p+1)}{3p},\ {\rm or}\\
	&v^3\in L^\infty([0,T^\ast],L^p(\Omega))\ {\rm with}\ p>\frac 7 2
\end{aligned}
\right.
$$
with $\Omega=(0,1)^3$ is a bounded domain in $\mathbb{R}^3$. In \cite{Zhou2010}, Zhou extended Cao's results to the entire space $\mathbb{R}^3$ and improved the corresponding supercritical exponents, that if
$$
v^3\in L^q([0,T^\ast],L^p(\mathbb{R}^3))\ {\rm with }\  p> \frac {10} 3,\ {\rm and}\ \frac 3 p+\frac 2q\leq\frac 3 4+\frac 1 {2p},
$$
or if
$$
\nabla v^3\in L^q([0,T^\ast],L^p(\Omega))\ {\rm with }\ \frac 3 p+\frac 2q\leq
\left\{
\begin{aligned}
	&\frac{19}{12}+\frac 1 {2s},\ \frac {30} {19}<p\leq 3,\\
	&\frac 3 2+\frac 3 {4s},\ 3<p\leq\infty,
\end{aligned}
\right.
$$
then $v$ is still regular at time $T^\ast<\infty$. Chea and Wolf \cite{Wolf2021} get the
regularity under the almost one-component Serrin condition:
$$
v^3\in L^q([0,T^\ast],L^p(\mathbb{R}^3))\ {\rm with }\ \frac 3 p +\frac 2 q<1,\ {\rm and}\ 3<p\leq\infty.
$$
And in \cite{Wu2024}, Wang, Wu and Zhang imposed stronger integrability assumptions in time and proved that the Leray weak solution is also regular in $\mathbb{R}^3\times (0,T)$, under the scaling-invariant Serrin condition imposed on one component of the velocity
$$
v^3\in L^{q,1}([0,T^\ast],L^p(\mathbb{R}^3))\ {\rm with }\  \frac 3 p +\frac 2 q\leq1,\ {\rm and}\ 3<p<\infty,
$$
here $L^{q,1}$ denotes the Lorentz space with respect to the variable $t$.

Chemin and Zhang initiated a program aimed at proving the regularity of the solutions by only imposing the following assumption in the critical spaces which are scaling invariant:
$$I_p(v\cdot e)\triangleq\int_0^{T^\ast}\Vert v\cdot e\Vert_{\dot{H}^{\frac1 2+\frac 2 p}}^p{\rm d}t<\infty$$
where $e\in\mathbb{S}^2$ and $p\in[2,\infty)$. They succeeded in the case $4<p<6$ in \cite{Chemin2016} by providing a technical decomposition of the horizontal direction $v^h$ in terms of the vertical direction of vorticity $\Omega$ and the vertical velocity component $v^3$. This result was later extended to $4<p<\infty$ in \cite{Zhang2017} by Chemin and Zhang. Subsequently, Lei \cite{Lei2019} gave a streamlined proof covering the full range $2\leq p<\infty$, thereby completing the program. Recently, Zhang \cite{Zhang2024} has made further extensions to the aforementioned results: define $\mathcal{D}(T)=\big\{\beta:[0,T)\rightarrow\mathbb{S}^2|\beta(t){\rm \ has \ finitely\ many\ jump\ discontinuities:\ } T_1,\dots,T_n\ {\rm on}\ (0,T)\ {\rm with}\ \beta'\in L^2(T_{i-1},T_i)\ {\rm for}\ i=1,\dots,n\big\}$, then for any $\beta(t)\in\mathcal{D}(T^\ast)$, if
$$\int_0^{T^\ast}\Vert v\cdot\beta(t)\Vert_{\dot{H}^{\frac 3 2}}^2{\rm d}t=\infty,$$
then $v$ blows up at finite time $T^\ast$.

The analysis in Miller \cite{Miller2020} is formulated in terms of the strain tensor $S=\left(\frac 1 2(\partial_iv^j+\partial_jv^i)\right)_{ij}$. The equation for the strain tensor $S$
$$\partial_t S+(v\cdot\nabla)S-\Delta S+S^2+\frac 1 4\Omega\otimes\Omega-\frac 1 4 |\Omega|^2 I_3+{\rm Hess}(P)=0$$
exhibits a far better algebraic structure than that of the vorticity equation, from which the following blow-up criterion is obtained: assume $\lambda_1(x)\leq\lambda_2(x)\leq\lambda_3(x)$ be the eigenvalues of the strain tensor $S$, let $\lambda^+_2(x)=\max\{\lambda_2(x),0\}$, if
$$
\lambda_2^+\in L^q([0,T^\ast],L^p(\mathbb{R}^3))\ {\rm with }\  \frac 2 q+\frac 3 p=2,\ {\rm and}\ \frac 3 2<p\leq\infty,
$$
then $v$ is still regular at time $T^\ast<\infty$. The aforementioned blow-up criterion readily yields the following criterion:
$$
\int_0^{T^\ast}\Vert\partial_3 v+\nabla v^3\Vert_{L^p}^q{\rm d}t=\infty\ {\rm with}\ \frac 2 q+\frac 3 p=2,\ {\rm and}\ \frac 3 2<p\leq\infty.
$$

Building upon the insights from the previous work, we propose the conjecture that the one component blow-up criterion remains valid under the $L^p$ norm framework: if
$$\int_0^{T^\ast}\Vert v\cdot e\Vert_{L^p}^{q}{\rm d}t<\infty,$$
with $\frac{2}{q}+\frac{3}{p}=1$, then $v$ is still regular at time $T^\ast<\infty$.

This hypothesis appears theoretically sound, as it naturally follows from the Sobolev embedding theorem based on the key results established in \cite{Chemin2016,Zhang2017,Lei2019}. However, our efforts to rigorously prove this blow-up criterion have encountered significant technical challenges. Specifically, when performing the anisotropic decomposition, we face a fundamental obstruction: the lack of regularity with the $v^3$ component prevents us from achieving the desired norm closure. This limitation results in insufficient regularity control over the remaining terms involving horizontal components, rendering them intractable for subsequent estimates.

To make progress toward resolving this conjecture, we adopt a strategic approach by first relaxing the conditions and establishing a weaker one-component blow-up criterion. We anticipate that this intermediate result may serve as a crucial stepping stone, potentially enabling a breakthrough in the eventual proof of the more stringent one-component blow-up criterion. Actually, we prove the following blow-up criterion:

\begin{theo}\label{main}
	Let $v_0\in\dot{H}^{\frac 1 2}$ with $\nabla\cdot v_0=0$ and $\Omega_0=\nabla \times v_0\in L^{r_0}$ for some $1<r_0<2$. Let $0<T^*<\infty$ and let
	$$v\in C([0,T^*);\dot{H}^{\frac 1 2})\cap L^2([0,T^*);\dot{H}^{\frac 3 2})$$
	be the unique local mild solution to the three-dimensional Navier-Stokes equations \eqref{NS} with initial data $v_0$.
	If for some $2<p<\infty$, $v$ satisfies that
	\begin{equation}\label{Lp}
		\int_0^{T^*}\left\Vert v\cdot e\right\Vert_{\dot{B}^{\frac 1 2+\frac 2 p}_{2,\infty}}^p{\rm d}t<\infty
	\end{equation}
	with $e\in\mathbb{S}^2$, then $v\in C([0,T^*];\dot{H}^{\frac 1 2})\cap L^2([0,T^*];\dot{H}^{\frac 3 2})$ and must be regular up to time $T^*$.
	
	More precisely
	$$\max\limits_{0\leq t\leq T^*}(\Vert v(t)\Vert_{\dot{H}^{\frac 1 2}}+\Vert\Omega(t)\Vert_{L^{r_0}})<\infty$$
	and for any $0<t_0<T^*$
	$$\max\limits_{t_0\leq t\leq T^*}(\Vert v(t)\Vert_{\dot{H}^{1}}+\Vert\nabla\Omega(t)\Vert_{L^{r_0}})<\infty.$$
\end{theo}

\begin{rema}
Since $\dot{H}^{\frac 1 2+\frac 2 p}\hookrightarrow \dot{B}^{\frac 1 2+\frac 2 p}_{2,\infty}$, it follows that our theorem can be viewed as an improvement of \cite{Chemin2016,Zhang2017,Lei2019}.
\end{rema}

Let us give a brief overview of the proof and explain some main steps. Without loss of generality, we assume $e=(0,0,1)$ throughout this paper and thus the blow-up criterion becomes
$$\int_0^{T^\ast}\left\Vert v^3\right\Vert_{\dot{B}^{\frac 1 2+\frac 2 p}_{2,\infty}}^p{\rm d}t$$

\noindent\textbf{Step 1} Anisotropic decomposition of the velocity.

The insightful method introduced in \cite{Zhang2017} and \cite{Lei2019} is to use the decomposition of the velocity field along horizontal and vertical directions and employ the two-dimensional vorticity $\omega$ and the vertical direction $v^3$ to represent the horizontal directions of velocity field. Denote
$$\nabla_{\rm h}=(\partial_1,\partial_2),\ \nabla_{\rm h}^{\bot}=(-\partial_2,\partial_1),\ \Delta_{\rm h}=\partial_1^2+\partial_2^2,$$
then by using Biot-Savart's law in the horizontal variables, we can use the vertical variable $\omega\triangleq-\partial_2 v^1+\partial_1 v^2$ of vorticity $\Omega$ and $v^3$ to express $v^{\rm h}$,
\begin{equation}\label{decomposition}
	v_{\rm curl}^{\rm h}=\nabla_{\rm h}^{\bot}\Delta_{\rm h}^{-1}\omega,\ v_{\rm div}^{\rm h}=\nabla_{\rm h}\Delta_{\rm h}^{-1}\partial_3 v^3, \ v^{\rm h}=v_{\rm curl}^{\rm h}-v_{\rm div}^{\rm h}=
	\left(\begin{array}{c}
		-\partial_2\Delta_{\rm h}^{-1}\omega-\partial_1\Delta_{\rm h}^{-1}\partial_3 v^3\\
		\partial_1\Delta_{\rm h}^{-1}\omega-\partial_2\Delta_{\rm h}^{-1}\partial_3 v^3
	\end{array}
	\right).		
\end{equation}

And it is easy to check that $\omega$ and $\nabla v^3$ solve below transport-diffusion equations respectively
\begin{equation}\label{crul}
	\partial_t\omega+v\cdot\nabla\omega-\Delta\omega=\partial_3 v^3\cdot\omega+\partial_2v^3\partial_3v^1-\partial_1v^3\partial_3v^2;
\end{equation}
\begin{equation}\label{v3}
	\partial_t\partial_i v^3+v\cdot\nabla\partial_i v^3-\Delta\partial_i v^3=-\partial_i v\cdot\nabla v^3+\partial_3\partial_i\Delta^{-1}\left(\sum_{l,m=1}^3\partial_m v^l\partial_l v^m\right),\ i=1,2,3.
\end{equation}
Thanks to Biot-Savart's law, the above system is equivalent to the original Navier-Stokes system \eqref{NS}.

\noindent\textbf{Step 2}
The dual argument and the refined estimate of $\left\Vert\omega_{r-1}\right\Vert_{\dot{B}^{s}_{r',1}}$.

As we have no a priori control on $\omega$, the quadratic terms $v^{\rm h}_{\rm curl}\cdot\nabla_{\rm h}\omega$ in \eqref{crul} cannot be controlled by any norm related to $v^3$. A standard strategy to circumvent this issue is to employ $L^r$ energy type estimate in conjunction with the divergence-free condition of $v$. Moreover,
to estimate term $\nabla_{\rm h}^{\bot}\Delta_{\rm h}^{-1}\omega$ and to control it by the $\Vert\omega_{\frac r 2}\Vert_2$ and $\Vert\nabla\omega_{\frac r 2}\Vert_2$ with the notation $a_s\triangleq\frac a {|a|}|a|^s\ (s\in]0,1[)$, the condition $r < 2$ is necessary.

For simplicity, we denote
$$\alpha(r)=\frac 1 r-\frac 1 2,\ \delta=3\alpha(r), \theta\in\left[0,\alpha(r)\right[$$
with $1<r<2$, and choosing $r$ sufficiently close to $2$. Inspired by \cite{Zhang2017}, we consider the estimate of $\partial_i v^3$ in an adapted Hilbert space. In order to close the estimates for the norms of $\omega$ and $v^3$, the scaling of $\mathcal{H}$ must be the same as that of $L^r$. Furthermore, the operator $\nabla_h\Delta_h^{-1}$ makes it natural to measure horizontal and vertical derivatives differently. This lead to the following definition.
\begin{defi}
	For $(s_1,s_2)\in\mathbb{R}^2$, $\dot{H}^{s_1,s_2}$ denotes the space of tempered distribution $a$ such that
	$$\left\Vert a\right\Vert_{\dot{H}^{s_1,s_2}}^2\triangleq\int\left|\xi_h\right|^{2s_1}\left|\xi_3\right|^{2s_2}\left|\hat{a}(\xi)\right|^2{\rm d}\xi$$
	with $\xi_{\rm h}=(\xi_1,\xi_2)$.
\end{defi}

We denote $\mathcal{H}^{\theta,r}\triangleq\dot{H}^{-3\alpha(r)+\theta,-\theta}$, and it is obvious that $\left\Vert|\nabla_h|^{-\delta}a\right\Vert_2$ is equivalent to $\Vert a\Vert_{\mathcal{H}^{\theta,r}}$ when $\theta=0$.

Note that we only have the condition $\|v^3\|_{L^p_T(\dot{B}^{\frac 1 2+\frac 2 p}_{2,\infty})}<\infty$. In the energy estimate of $\omega$ for the $L^r-$norm, we need to use the dual method and encounter difficulties in estimating the term $\left\Vert\omega_{r-1}\right\Vert_{\dot{B}^{s}_{r',1}}$ which cannot be directly obtained a suitable control from the Besov norm estimate in Lemma 3.2 in \cite{Zhang2017}. The key innovation of our approach lies in the novel integration of the above Besov norm estimate for $\omega_{\frac r 2}$ composited with the $\frac 2 {r'}$-order H\"{o}lder continuous function $G$ and a more refined Besov norm interpolation technique. Therefore, we arrive at the estimate
$$\left\Vert G\left(\omega_{\frac r 2}\right)\right\Vert_{\dot{B}^{s}_{r',1}}\leq\left\Vert\nabla\omega_{\frac r 2}\right\Vert_2^s\cdot\left\Vert\omega_{\frac r 2}\right\Vert_2^{\frac 2 {r'}-s}$$
for all $s\in\left]0,\frac 2 {r'}\right[$. This powerful combination successfully overcomes the fundamental limitation outlined earlier.

\noindent\textbf{Step 3} Estimate of $(\omega, v^3)$ for the case $p>4$.

For $p> 4$, since $\frac 1 2+\frac 2 p<1$ it follows that the quantity $\left\Vert v^3\right\Vert_{\dot{B}^{\frac 1 2+\frac 2 p}_{2,\infty}}$ cannot be obtained through the interpolation of the $\dot{H}^s$-norm of $v^3$, and the lack of regularity in $\dot{B}^{\frac 1 2+\frac 2 p}_{2,\infty}$ precludes the application of the method in \cite{Lei2019}, which is based on a straightforward combination of H\"{o}lder's inequality and Sobolev embedding, we therefore refine the estimation of detailed anisotropic inner products for equation \eqref{v3}, as introduced by Chemin et al. in \cite{Zhang2017} and stated in Lemma 5.2, which mainly relies on the Bony decomposition and the commutator estimate.

\noindent\textbf{Step 4} Estimate of $(\omega, v^3)$ for the case $2< p\leq4$.

Diverging from the case of $p> 4$, the improved regularity of
$\left\Vert v^3\right\Vert_{\dot{B}^{\frac 1 2+\frac 2 p}_{2,\infty}}$ enables a direct application of the Bony decomposition estimate under the dual method. Therefore, we reduce to the norm $\mathcal{H}^{\theta,r}$ with $\theta=0$, and the anisotropic inner product estimate of $\partial_i v^3$ can be reduced to an $L^2$ inner product estimate of $\left|\nabla_{\rm h}\right|^{-\delta}\partial_i v^3$. However, for $2<p\leq4$, we cannot control the term $\Vert T(\partial_{\rm h} v^3,\omega_{r-1})\Vert_{L^{r'}_{\rm v}L^{\frac{2r}{3r-2}}_{\rm h}}$ by $\left\Vert v^3\right\Vert_{\dot{B}^{\frac 1 2+\frac 2 p}_{2,\infty}}$, $\left\Vert\omega_{\frac r 2}\right\Vert_2$ and $\left\Vert\nabla\omega_{\frac r 2}\right\Vert_2$. Furthermore, we have not embedding from isotropic Besov space to $L^2_{\rm h}(\dot{B}^{s}_{2,\infty})$ for any $s>0$, therefore we adjust the index assignment issue of the anisotropic Bony decomposition and utilize the embedding inequalities of $(\dot{B}^{s_1}_{p,\infty})_{\rm v}(\dot{B}^{s_2}_{p,1})_{\rm h}$ and $(\dot{B}^{s_1}_{p,\infty})_{\rm h}(\dot{B}^{s_2}_{p,1})_{\rm v}$ with $s_i>0$ to obtain the desired blow-up criterion, this requires more computational effort compared to \cite{Zhang2017}.

In contrast to the methods employed by \cite{Zhang2017} and \cite{Lei2019}, we develop a refined approach that incorporates the estimates of the dual norm
$$\left\Vert\left(2^{ks}\Vert\Delta_k^{\rm h}f\Vert_p\right)_k\right\Vert_{l^q(\mathbb{Z})},\ \left\Vert\left(2^{-ks}\Vert\Delta_k^{\rm h}f\Vert_{p'}\right)_k\right\Vert_{l^{q'}(\mathbb{Z})},$$ $$\left\Vert\left(2^{ls}\Vert\Delta_l^{\rm v}f\Vert_p\right)_l\right\Vert_{l^q(\mathbb{Z})},\ \left\Vert\left(2^{-ls}\Vert\Delta_l^{\rm v}f\Vert_{p'}\right)_l\right\Vert_{l^{q'}(\mathbb{Z})}$$
with $\frac 1 p+\frac 1 {p'}=1$, $\frac 1 q+\frac 1 {q'}=1$ and $s\in\mathbb{R}$. Combining with the embedding relationships between the norm
$$\left\Vert\left(2^{ks'}\Vert\Delta_k^{\rm h}f\Vert_p\right)_k\right\Vert_{l^q(\mathbb{Z})},\ \left\Vert\left(2^{ls'}\Vert\Delta_l^{\rm v}f\Vert_p\right)_l\right\Vert_{l^q(\mathbb{Z})}\ {\rm with}\ s'>0$$
and the isotropic Besov space $\dot{B}^{s'}_{p,q}$, we can systematically overcome the regularity limitations.

Based on the method in \cite{Lei2019}, we employ a dual-space frame work that critically leverages both duality properties and the embedding relationship among anisotropic Besov spaces, isotropic Besov spaces, and $L^q$ spaces. Utilizing these computational strategies allows us to successfully overcome the restrictions of Sobolev embeddings, and moreover the problem can be reduced to performing an isotropic Bony decomposition on the horizontal or vertical direction. Through the careful selection of appropriate indices, this approach enables us to rigorously derive the desired norm estimates.

\section{Notations and Preliminaries}
	
\quad In this section, we recall estimates that will be useful  later and the basics of anisotropic Littlewood-Paley theory.

In order to use the divergence free condition on $v$, we perform a $L^r$ energy estimate for $\omega$. This is based on the following lemma:
		
\begin{lemm}\label{L-norm estiamte}\cite{Chemin2016}
	Let $r\in ]1,2]$ and $a_0$ a function in $L^r$. Let us consider a function $f\in L^1_{\rm{loc}}(\mathbb{R}^+;L^r)$ and $v$ a divergence free vector field in $L^2_{\rm{loc}}(\mathbb{R}^+;L^\infty)$. If $a$ solves
	\begin{equation}\label{L-r}
		\left\{
		\begin{array}{l}
			\partial_t a-\Delta a+v\cdot \nabla a=f,\\
			a|_{t=0}=a_0,
		\end{array}
		\right.
	\end{equation}
	then $|a|^{r/2}$ belongs to $L^\infty_{loc}(\mathbb{R}^+;L^2)\cap L^2_{loc}(\mathbb{R}^+;\dot{H}^1)$ and
	\begin{equation}
		\begin{aligned}
			&\frac 1 r \int_{\mathbb{R}^3}|a(t,x)|^r{\rm d}x+(r-1)\int_0^t\int_{\mathbb{R}^3}|\nabla a(t',x)|^2|a(t',x)|^{r-2}{\rm d}x{\rm d}t'\\
			=&\frac 1 r \int_{\mathbb{R}^3}|a_0(x)|^r{\rm d}x+\int_0^t\int_{\mathbb{R}^3}f(t',x)a(t',x)|a(t',x)|^{r-2}{\rm d}x{\rm d}t'.
		\end{aligned}
	\end{equation}
\end{lemm}

The below Sobolev type inequalities which involves the regularities of $a_{\frac r 2}$ and$\nabla a_{\frac r 2}$ in $L^2$ is highly useful in subsequent estimation.

\begin{lemm}\label{Sobolev type}\cite{Zhang2017}
	For $r\in]3/2,2[$, we have
	\begin{equation}\label{nabla L-r}
		\Vert\nabla a\Vert_r\lesssim\left\Vert\nabla a_{\frac r 2}\right\Vert_2\left\Vert a_{\frac r 2}\right\Vert_2^{\frac 2 r-1}.
	\end{equation}
	Moreover, for $s\in[-3\alpha(r),1-\alpha(r)]$, we have
	\begin{equation}\label{nabla H-s}
		\Vert\nabla a\Vert_{\dot{H}^s}\lesssim\left\Vert\nabla a_{\frac r 2}\right\Vert_2^{3\alpha(r)+s}\left\Vert a_{\frac r 2}\right\Vert_2^{1-\alpha(r)-s}.
	\end{equation}
\end{lemm}

Since we shall use the anisotropic Littlewood-Paley theory, we recall the functional space framework used in this section. From
\cite{BCD}, the definition and some basic properties of Besov spaces can be seen in Appendix A. In order to consider the product of a distribution in the anisotropic Besov space, the embeddings from anisotropic Besov spaces into isotropic ones are highly useful:
\begin{lemm}\cite{Chemin2016}\label{embed1}
	Let $s>0$ and $(p,q)\in [1,\infty]^2$, if $p\geq q$ we have
	$$\Vert a\Vert_{L^p_{\rm h}\left(\dot{B}^s_{p,q}\right)_{\rm v}}\lesssim\Vert a\Vert_{\dot{B}^s_{p,q}}\ {\rm and}\ \Vert a\Vert_{L^p_{\rm v}\left(\dot{B}^s_{p,q}\right)_{\rm h}}\lesssim\Vert a\Vert_{\dot{B}^s_{p,q}}.$$
\end{lemm}

\begin{lemm}\cite{Chemin2016}\label{embed2}
	For any $s>0$ and any $\theta\in]0,s[$, we have
	$$\Vert a\Vert_{\left(\dot{B}^{s-\theta}_{p,q}\right)_{\rm h}\left(\dot{B}^\theta_{p,1}\right)_{\rm v}}\lesssim\Vert a\Vert_{\dot{B}^s_{p,q}}\ {\rm and}\ \Vert a\Vert_{\left(\dot{B}^{s-\theta}_{p,q}\right)_{\rm v}\left(\dot{B}^\theta_{p,1}\right)_{\rm h}}\lesssim\Vert a\Vert_{\dot{B}^s_{p,q}}.$$
\end{lemm}

\begin{lemm}\cite{BCD}\label{embed3}
	For any $p\in[2,\infty)$, then $\dot{B}^0_{p,2}\hookrightarrow L^p$ and $L^{p'}\hookrightarrow\dot{B}^0_{p',2}$. For any $p\in[1,2]$, then $\dot{B}^0_{p,p}\hookrightarrow L^p$ and $L^{p'}\hookrightarrow\dot{B}^0_{p',p'}$.
\end{lemm}

Directly following the proof of the Proposition 3.1 in \cite{Zhang2017}, we readily obtain that the index $\beta$ can be extended to the case $\beta=1/2$ with $r$ sufficiently close to $2$, hence we have the following key estimate:

\begin{prop}\label{anisoropic v^h}
	Let $v$ be a divergence free vector field. Let us consider $\theta\in]0,3\alpha(r)[$ and $\beta\in]0,1/2]$. Then we have
	$$\left\Vert v^{\rm h}\right\Vert_{\left(\dot{B}^{1}_{2,1}\right)_{\rm h}\left(\dot{B}^{1-3\alpha(r)-\beta}_{2,1}\right)_{\rm v}}\lesssim\left\Vert\omega_{\frac r 2}\right\Vert_2^{2\alpha(r)+\beta}\left\Vert\nabla\omega_{\frac r 2}\right\Vert_2^{1-\beta}+\left\Vert\partial_3v^3\right\Vert_{\mathcal{H}^{\theta,r}}^\beta\left\Vert\nabla\partial_3v^3\right\Vert_{\mathcal{H}^{\theta,r}}^{1-\beta}
	.$$
\end{prop}

And following the same line of proof of Lemma 4.5 of \cite{Chemin2016} we have the following product laws in the anisotropic Besov spaces:

\begin{lemm}\label{anisotropic product}
	Let $q,\ q_1,\ q_2\geq1$ with $1/q_1+1/q_2\geq 1/q$, $p,\ p_1,\ p_2\geq1$ with $1/p\leq1/p_1+1/p_2\leq1$. Then for $s_1+s_2>0$ with $s_1<2\min(\frac 1 {p_1},\frac 1 {p_1}+\frac 1 {p_2}-\frac 1 p)$, $s_2<2\min(\frac 1 {p_2},\frac 1 {p_1}+\frac 1 {p_2}-\frac 1 p)$ and $\sigma_1+\sigma_2>0$ with $\sigma_1<\min(\frac 1 {p_1},\frac 1 {p_1}+\frac 1 {p_2}-\frac 1 p)$, $\sigma_2<\min(\frac 1 {p_2},\frac 1 {p_1}+\frac 1 {p_2}-\frac 1 p)$, we obtain the inequality of the product $ab$
	$$
	\left\Vert a b\right\Vert_{\left(\dot{B}^{s_1+s_2-2\left(\frac 1 {p_1}+\frac 1 {p_2}-\frac 1 p\right)}_{p,q}\right)_{\rm h}\left(\dot{B}^{\sigma_1+\sigma_2-\left(\frac 1 {p_1}+\frac 1 {p_2}-\frac 1 p\right)}_{p,q}\right)_{\rm v}}\lesssim
	\left\Vert a\right\Vert_{\left(\dot{B}^{s_1}_{p_1,q_1}\right)_{\rm h}\left(\dot{B}^{\sigma_1}_{p_1,q_1}\right)_{\rm v}}\left\Vert b\right\Vert_{\left(\dot{B}^{s_2}_{p_2,q_2}\right)_{\rm h}\left(\dot{B}^{\sigma_2}_{p_2,q_2}\right)_{\rm v}}
	.$$
	Moreover, if $q_i=1$ we can take $s_i=2\min(\frac 1 {p_i},\frac 1 {p_1}+\frac 1 {p_2}-\frac 1 p)$ and $\sigma_i=\min(\frac 1 {p_i},\frac 1 {p_1}+\frac 1 {p_2}-\frac 1 p)$.
\end{lemm}

The description of regularity of $\omega_{r-1}$ in terms of Besov spaces will be useful. In order to get a refine estimate of $\omega_{r-1}$, the following lemma is important:

\begin{lemm}\cite{Zhang2017}\label{Holder}
	Let $(s,\alpha)\in]0,1[^2$ and $(p,q)\in[1,\infty]^2$. We consider a function $G$ from $\mathbb{R}$ to $\mathbb{R}$ which is a H\"{o}lderian of exponent $\alpha$. Then for any $a\in\dot{B}^s_{p,q}$, one has
	$$\left\Vert G(a)\right\Vert_{\dot{B}^{\alpha s}_{\frac p \alpha,\frac q \alpha}}\lesssim\Vert G\Vert_{C^\alpha}\Vert a\Vert_{\dot{B}^s_{p,q}}^\alpha\ with\ \Vert G\Vert_{C^\alpha}\triangleq\sup_{r\neq r'}\frac{|G(r)-G(r')|}{|r-r'|^\alpha}.$$
\end{lemm}

Based on Lemma \ref{Holder}, we obtain the Sobolev type refined inequality of $\omega_{r-1}$.

\begin{prop}\label{refine Sobolev}
	For any $s\in]0,\alpha[$ with $\alpha\in]0,1[$, $G$ is $\alpha$-order H\"{o}lder continuous, then we have
	$$\Vert G(\omega_{\frac r 2})\Vert_{\dot{B}^s_{\frac 2 \alpha,1}}\lesssim\left\Vert\nabla\omega_{\frac r 2}\right\Vert_2^s\left\Vert\omega_{\frac r 2}\right\Vert_2^{\alpha-s}.$$
	In particular, if $s\in]0,\frac 2 {r'}[$, we have
	$$\Vert \omega_{r-1}\Vert_{\dot{B}^s_{r',1}}\lesssim\left\Vert\nabla\omega_{\frac r 2}\right\Vert_2^s\left\Vert\omega_{\frac r 2}\right\Vert_2^{\frac 2 {r'}-s}.$$
\end{prop}
\begin{proof}
	Since $s\in]0,\alpha[$, we can choose suitable $s'\in]s,\alpha[$, according to Lemma \ref{Holder}, we have
	$$
	\begin{aligned}
		&\left\Vert G(\omega_{\frac r 2})\right\Vert_{\dot{B}^s_{\frac 2 \alpha,1}}\\
		=&\sum_{j\leq N}2^{js}\left\Vert\Delta_jG(\omega_{\frac r 2})\right\Vert_{\frac 2 \alpha}+\sum_{j> N}2^{js}\left\Vert\Delta_jG(\omega_{\frac r 2})\right\Vert_{\frac 2 \alpha}=\sum_{j\leq N}2^{js}\left\Vert G(\omega_{\frac r 2})\right\Vert_{\frac 2 \alpha}+\sum_{j> N}2^{j(s-s')}2^{js'}\left\Vert\Delta_jG(\omega_{\frac r 2})\right\Vert_{\frac 2 \alpha}\\
		\lesssim&2^{Ns}\left\Vert\omega_{\frac r 2}\right\Vert_2^{\alpha}+\sum_{j> N}2^{j(s-s')}\left\Vert G(\omega_{\alpha})\right\Vert_{\dot{B}^s_{\frac 2 \alpha,\infty}}\lesssim 2^{Ns}\left\Vert\omega_{\frac r 2}\right\Vert_2^{\frac 2 {r'}}+2^{N(s-s')}\left\Vert\omega_{\frac r 2}\right\Vert_{\dot{B}^{\frac {s'} \alpha}_{2,\infty}}^{\alpha}\\
		\lesssim&2^{Ns}\left\Vert\omega_{\frac r 2}\right\Vert_2^{\alpha}+2^{N(s-s')}\left\Vert \omega_{\frac r 2}\right\Vert_2^{\alpha-s'}\left\Vert\nabla\omega_{\alpha}\right\Vert_2^{s'}.
	\end{aligned}
	$$
	Choosing $N$ such that $2^N\sim\left\Vert \nabla\omega_{\frac r 2}\right\Vert_2/\left\Vert\omega_{\frac r 2}\right\Vert_2$, finally we obtain that
	$$\Vert G(\omega_{\frac r 2})\Vert_{\dot{B}^s_{\frac 2 \alpha,1}}\lesssim\left\Vert\nabla\omega_{\frac r 2}\right\Vert_2^s\left\Vert\omega_{\frac r 2}\right\Vert_2^{\alpha-s}.$$
	
	In particular, since $\omega_{r-1}=G(\omega_{\frac r 2})$ with $G(z)=z|z|^{-2\alpha(r)}$, it is easy to check that $G$ is $2/r'$-order H\"{o}lder. Applying above result with $\alpha=\frac 2 {r'}$, thus the second inequality is proved.
\end{proof}

\section{Estimate of $(\omega,v^3)$: Case $p>4$}
\subsection{Estimate of $\omega$}
\quad Recall the equation of $\omega$ \eqref{crul}
$$
\partial_t\omega+v\cdot\nabla\omega-\Delta\omega=\partial_3 v^3\cdot\omega+\partial_2v^3\partial_3v^1-\partial_1v^3\partial_3v^2.
$$
By virtue of Lemma \ref{L-norm estiamte}, we obtain
\begin{equation}\label{curl L^r}
	\begin{aligned}
		&\frac 1 r\frac{\rm d}{{\rm d}t}\Vert\omega_{\frac r 2}\Vert_2^2+\frac{4(r-1)}{r^2}\Vert\nabla\omega_{\frac r 2}\Vert_2^2\\
		=&\int\partial_3v^3|\omega|^r{\rm d}x+\int\partial_2v^3\partial_3v^1\omega_{r-1}{\rm d}x-\int\partial_1v^3\partial_3v^2\omega_{r-1}{\rm d}x=: I_1+I_2+I_3.
	\end{aligned}
\end{equation}
For $I_1$, take $r$ sufficiently close to $2$, then for all $ p\in[2,\infty[$ according to isotropic Bony's decomposition and Lemma \ref{Sobolev type} we easily deduce that
\begin{equation}\label{I1}
	\begin{aligned}
		|I_1|&\lesssim\Vert\partial_3v^3\Vert_{\dot{B}^{-\frac 1 2+\frac 2 p}_{2,\infty}}\Vert(\omega_{\frac r 2})^2\Vert_{\dot{B}^{\frac 1 2-\frac 2 p}_{2,1}}\lesssim\Vert v^3\Vert_{\dot{B}^{\frac 1 2+\frac 2 p}_{2,\infty}}\Vert \omega_{\frac r 2}\Vert_{\dot{B}^{1-\frac 2 p}_{2,2}}\Vert\nabla\omega_{\frac r 2}\Vert_2\\
		&\lesssim\Vert v^3\Vert_{\dot{B}^{\frac 1 2+\frac 2 p}_{2,\infty}}\Vert\omega_{\frac r 2}\Vert_2^{2-\frac 2 p}\Vert\nabla\omega_{\frac r 2}\Vert_2^{\frac 2 p}.
	\end{aligned}
\end{equation}
The estimate of $I_3$ is similar to $I_2$ and therefore will be omitted. $I_2$ can be written as
$$I_2=-\int\partial_2v^3\partial_1\Delta_{\rm h}^{-1}\partial_3^2v^3\cdot\omega_{r-1}{\rm d}x-\int\partial_2v^3\partial_2\Delta_{\rm h}^{-1}\partial_3\omega\cdot\omega_{r-1}{\rm d}x,$$
the problem can be concluded by below two terms
\begin{equation}\label{I21 and I22}
	I_{21}:=\int\partial_hv^3|\nabla_{\rm h}|^{-1}\partial_3^2v^3\cdot\omega_{r-1}{\rm d}x\ {\rm and}\ I_{22}:=\int\partial_hv^3|\nabla_{\rm h}|^{-1}\partial_3\omega\cdot\omega_{r-1}{\rm d}x.
\end{equation}
Using Bony's decomposition and the Leibnitz formula, we write
$$
\begin{aligned}
	\partial_hv^3\cdot\omega_{r-1}&=T(\partial_h v^3,\omega_{r-1})+R(\partial_h v^3,\omega_{r-1})+T(\omega_{r-1},\partial_h v^3)\\
	&=\partial_{\rm h}T(\omega_{r-1},v^3)+A(v^3,\omega)\ {\rm with}\\
	A(v^3,\omega)&=T(\partial_h v^3,\omega_{r-1})+R(\partial_h v^3,\omega_{r-1})-T(\partial_h \omega_{r-1},v^3).
\end{aligned}
$$

\noindent$\bullet$ Estimate of $I_{21}$.\\
Integrating by parts, we have
$$I_{21}=\int|\nabla_{\rm h}|^{-1}\partial_3^2v^3\cdot A(v^3,\omega){\rm d}x-\int\partial_{\rm h}|\nabla_{\rm h}|^{-1}\partial_3^2v^3\cdot T(\omega_{r-1},v^3){\rm d}x.$$
According to Proposition \ref{refine Sobolev} for all $p> 4$, we have
\begin{equation}\label{T(w,a) estimate}
	\begin{aligned}
		\Vert T(\omega_{r-1},v^3)\Vert_{\dot{H}^{3\alpha(r)}}&\lesssim\Vert\omega_{r-1}\Vert_{\dot{B}^{\frac 3 r-2-\frac 2 p}_{\infty,2}}\Vert v^3\Vert_{\dot{B}^{\frac 1 2+\frac 2 p}_{2,\infty}}\lesssim \Vert v^3\Vert_{\dot{B}^{\frac 1 2+\frac 2 p}_{2,\infty}}\Vert\omega_{r-1}\Vert_{\dot{B}^{1-\frac 2 p}_{r',2}} \\
		&\lesssim\Vert v^3\Vert_{\dot{B}^{\frac 1 2+\frac 2 p}_{2,\infty}}\Vert\nabla\omega_{\frac r 2}\Vert_2^{1-\frac 2 p}\Vert\omega_{\frac r 2}\Vert_2^{\frac 2 {r'}-1+\frac 2 p}
	\end{aligned}
\end{equation}
Choosing $\varepsilon>0$ small enough, by Lemma \ref{embed2} $\dot{B}^{\frac 1 2-3\varepsilon}_{(\frac 3 2-\frac 1 r-\varepsilon)^{-1},1}\hookrightarrow(\dot{B}^{\frac 1 r-\frac 1 2-\theta-2\varepsilon}_{(\frac 3 2-\frac 1 r-\varepsilon)^{-1},1})_{\rm h}(\dot{B}^{1-\frac 1 r+\theta-\varepsilon}_{(\frac 3 2-\frac 1 r-\varepsilon)^{-1},1})_{\rm v}\hookrightarrow\dot{H}^{-1+3\alpha(r)-\theta,\theta}$, then for all $p> 4$,
\begin{equation}\label{A(a,w) estimate}
	\begin{aligned}
		&\Vert A(v^3,\omega)
		\Vert_{\dot{B}^{\frac 1 2-3\varepsilon}_{(\frac 3 2-\frac 1 r-\varepsilon)^{-1},1}}\\
		\leq& \Vert T(\partial_h v^3,\omega_{r-1})\Vert_{\dot{B}^{\frac 1 2-3\varepsilon}_{(\frac 3 2-\frac 1 r-\varepsilon)^{-1},1}}+\Vert R(\partial_h v^3,\omega_{r-1})\Vert_{\dot{B}^{\frac 1 2-3\varepsilon}_{(\frac 3 2-\frac 1 r-\varepsilon)^{-1},1}}+\Vert T(\partial_h \omega_{r-1},v^3)\Vert_{\dot{B}^{\frac 1 2-3\varepsilon}_{(\frac 3 2-\frac 1 r-\varepsilon)^{-1},1}}\\
		\lesssim&\Vert\partial_{\rm h}v^3\Vert_{\dot{B}^{-\frac 1 2+\frac 2 p-3\varepsilon}_{(\frac 1 2-\varepsilon)^{-1},\infty}}\Vert \omega_{r-1}\Vert_{\dot{B}^{1-\frac 2 p}_{r',1}}+\Vert\partial_{\rm h}\omega_{r-1}\Vert_{\dot{B}^{-\frac 2 p}_{r',1}}\Vert v^3\Vert_{\dot{B}^{\frac 1 2+\frac 2 p-3\varepsilon}_{(\frac 1 2-\varepsilon)^{-1},\infty}}\lesssim\Vert v^3\Vert_{\dot{B}^{\frac 1 2+\frac 2 p}_{2,\infty}}\Vert\omega_{r-1}\Vert_{\dot{B}^{1-\frac 2 p}_{r',1}}\\
		\lesssim&\Vert v^3\Vert_{\dot{B}^{\frac 1 2+\frac 2 p}_{2,\infty}}\Vert\nabla\omega_{\frac r 2}\Vert_2^{1-\frac 2 p}\Vert\omega_{\frac r 2}\Vert_2^{\frac 2 {r'}-1+\frac 2 p}.
	\end{aligned}
\end{equation}
Therefore by \eqref{T(w,a) estimate} and \eqref{A(a,w) estimate} we infer that
\begin{equation}\label{I21}
	\begin{aligned}
		|I_{21}|&\lesssim\Vert\partial_3^2 v^3\Vert_{\mathcal{H}^{\theta,r}}\Vert T(\omega_{r-1},v^3)\Vert_{\dot{H}^{3\alpha(r)-\theta,\theta}}+\Vert|\nabla_{\rm h}^{-1}\partial_3^2 v^3\Vert_{\dot{H}^{1-3\alpha(r)+\theta,-\theta}}\Vert A(v^3,\omega)\Vert_{\dot{H}^{-1+3\alpha(r)-\theta,\theta}}\\
		&\lesssim \Vert v^3\Vert_{\dot{B}^{\frac 1 2+\frac 2 p}_{2,\infty}}\Vert\nabla\omega_{\frac r 2}\Vert_2^{1-\frac 2 p}\Vert\omega_{\frac r 2}\Vert_2^{\frac 2 {r'}-1+\frac 2 p}\Vert\nabla\partial_3 v^3\Vert_{\mathcal{H}^{\theta,r}}.
	\end{aligned}
\end{equation}

\noindent$\bullet$ Estimate of $I_{22}$.\\
Integrating by parts, we have
$$I_{22}=\int|\nabla_{\rm h}|^{-1}\partial_3\omega\cdot A(v^3,\omega){\rm d}x-\int\partial_{\rm h}|\nabla_{\rm h}|^{-1}\partial_3\omega\cdot T(\omega_{r-1},v^3){\rm d}x.$$
According to Proposition \ref{refine Sobolev} for all $p> 4$, we obtain
\begin{equation}\label{T(w,a) L-norm estimate}
	\begin{aligned}
		\Vert T(\omega_{r-1},v^3)\Vert_{r'}&\leq
		\Vert T(\omega_{r-1},v^3)\Vert_{\dot{B}^0_{r',1}}
		\lesssim\Vert\omega_{r-1}\Vert_{\dot{B}^{\frac 3 r-2-\frac 2 p}_{\infty,1}}\Vert v^3\Vert_{\dot{B}^{2-\frac 3 r+\frac 2 p}_{r',\infty}}\lesssim \Vert v^3\Vert_{\dot{B}^{\frac 1 2+\frac 2 p}_{2,\infty}}\Vert\omega_{r-1}\Vert_{\dot{B}^{1-\frac 2 p}_{r',2}} \\
		&\lesssim\Vert v^3\Vert_{\dot{B}^{\frac 1 2+\frac 2 p}_{2,\infty}}\Vert\nabla\omega_{\frac r 2}\Vert_2^{1-\frac 2 p}\Vert\omega_{\frac r 2}\Vert_2^{\frac 2 {r'}-1+\frac 2 p}.
	\end{aligned}
\end{equation}
Since $\dot{B}^{\frac 1 2}_{(\frac 3 2-\frac 1 r)^{-1},1}\hookrightarrow L^{(\frac 3 2-\frac 1 r)^{-1}}_{\rm h}(\dot{B}^{\frac 1 2}_{(\frac 3 2-\frac 1 r)^{-1},1})_{\rm v}\hookrightarrow L^{(\frac 3 2-\frac 1 r)^{-1}}_{\rm h}(\dot{B}^{0}_{r',1})_{\rm v}\hookrightarrow L^{r'}_{\rm v}L^{(\frac 3 2-\frac 1 r)^{-1}}_{\rm h}$, similarly we have for $p>4$
\begin{equation}\label{A(a,w) L-norm estimate}
	\begin{aligned}
		&\Vert A(v^3,\omega)
		\Vert_{\dot{B}^{\frac 1 2}_{(\frac 3 2-\frac 1 r)^{-1},1}}\\
		\lesssim&\Vert\partial_{\rm h}v^3\Vert_{\dot{B}^{-\frac 1 2+\frac 2 p}_{2,\infty}}\Vert \omega_{r-1}\Vert_{\dot{B}^{1-\frac 2 p}_{r',1}}+\Vert\partial_{\rm h}\omega_{r-1}\Vert_{\dot{B}^{-\frac 2 p}_{r',1}}\Vert v^3\Vert_{\dot{B}^{\frac 1 2+\frac 2 p}_{2,\infty}}\lesssim\Vert v^3\Vert_{\dot{B}^{\frac 1 2+\frac 2 p}_{2,\infty}}\Vert\omega_{r-1}\Vert_{\dot{B}^{1-\frac 2 p}_{r',1}}\\
		\lesssim&\Vert v^3\Vert_{\dot{B}^{\frac 1 2+\frac 2 p}_{2,\infty}}\Vert\nabla\omega_{\frac r 2}\Vert_2^{1-\frac 2 p}\Vert\omega_{\frac r 2}\Vert_2^{\frac 2 {r'}-1+\frac 2 p}.
	\end{aligned}
\end{equation}
Therefore by Lemma \ref{Sobolev type}, \eqref{T(w,a) L-norm estimate} and \eqref{A(a,w) L-norm estimate} we infer that for $p>4$
\begin{equation}\label{I22 p>4}
	\begin{aligned}
		|I_{22}|&\lesssim\Vert\partial_3\omega\Vert_r\Vert T(\omega_{r-1},v^3)\Vert_{r'}+\Vert |\nabla_{\rm h}|^{-1}\partial\omega\Vert_{L^r_{\rm v}L^{(\frac 1 r-\frac 1 2)^{-1}}_{\rm h}}\Vert A(v^3,\omega)\Vert_{L^{r'}_{\rm v}L^{(\frac 3 2-\frac 1 r)^{-1}}_{\rm h}}\\
		&\lesssim \Vert v^3\Vert_{\dot{B}^{\frac 1 2+\frac 2 p}_{2,\infty}}\Vert\nabla\omega_{\frac r 2}\Vert_2^{1-\frac 2 p}\Vert\omega_{\frac r 2}\Vert_2^{\frac 2 {r'}-1+\frac 2 p}\Vert\partial\omega\Vert_r\lesssim\Vert v^3\Vert_{\dot{B}^{\frac 1 2+\frac 2 p}_{2,\infty}}\Vert\nabla\omega_{\frac r 2}\Vert_2^{2-\frac 2 p}\Vert\omega_{\frac r 2}\Vert_2^{\frac 2 p}.
	\end{aligned}
\end{equation}
Subsuming the estimates \eqref{I1}, \eqref{I21} and \eqref{I22 p>4} into \eqref{curl L^r}, we obtain
\begin{equation}\label{w estimate}
	\begin{aligned}
		&\frac 1 r\frac{\rm d}{{\rm d}t}\Vert\omega_{\frac r 2}\Vert_2^2+\frac{4(r-1)}{r^2}\Vert\nabla\omega_{\frac r 2}\Vert_2^2\\
		\leq&\Vert v^3\Vert_{\dot{B}^{\frac 1 2+\frac 2 p}_{2,\infty}}(\Vert\nabla\omega_{\frac r 2}\Vert_2^{2-\frac 2 p}\Vert\omega_{\frac r 2}\Vert_2^{\frac 2 p}+\Vert\nabla\omega_{\frac r 2}\Vert_2^{1-\frac 2 p}\Vert\omega_{\frac r 2}\Vert_2^{\frac 2 {r'}-1+\frac 2 p}\Vert\nabla\partial_3 v^3\Vert_{\mathcal{H}^{\theta,r}}).
	\end{aligned}
\end{equation}

\subsection{Estimate of $v^3$}
\quad Recall the equation of $v^3$ \eqref{v3}
$$
\partial_t\partial_3 v^3+v\cdot\nabla\partial_3 v^3-\Delta\partial_3 v^3=-\partial_3 v\cdot\nabla v^3+\partial_3\partial_3\Delta^{-1}\left(\sum_{l,m=1}^3\partial_m v^l\partial_l v^m\right).
$$
Taking the $\mathcal{H}^{\theta,r}$ inner product of the $\partial_3v^3$ equation yields that
\begin{equation}\label{v3 H-norm}
	\begin{aligned}
		&\frac 1 2\frac {\rm d} {{\rm d}t}\Vert \partial_3v^3\Vert_{\mathcal{H}^{\theta,r}}^2+\Vert \nabla\partial_3v^3\Vert_{\mathcal{H}^{\theta,r}}^2=-\sum_{n=1}^3\left(Q_n(v,v)|\partial_3v^3\right)_{\mathcal{H}^{\theta,r}}\ {\rm with}\\
		&Q_1(v,v)\triangleq({\rm Id}+\partial_3^2\Delta^{-1})(\partial_3v^3)^2+\partial_3^2\Delta^{-1}\left(\sum_{l,m=1}^2\partial_lv^m\partial_mv^l\right),\\
		&Q_2(v,v)\triangleq({\rm Id}+2\partial_3^2\Delta^{-1})\left(\sum_{l=1}^2\partial_lv^3\partial_3v^l\right) {\rm and}\ Q_3(v,v)\triangleq v\cdot\nabla\partial_3 v^3.
	\end{aligned}
\end{equation}
Moreover we have for any $\delta_1,\delta_2\in\mathbb{R}$
$$
\begin{aligned}
	(a|b)_{\mathcal{H}^{\theta,r}}&=\sum_{k,l\in\mathbb{Z}}2^{2k(-3\alpha(r)+\theta)}2^{-2l\theta}(\Delta_{k}^{\rm h}\Delta_l^{\rm v}a|\Delta_{k}^{\rm h}\Delta_l^{\rm v}b)\\
	&\leq\Vert a \Vert_{(\dot{B}^{-3\alpha(r)+\theta+\delta_1}_{2,\infty})_{\rm h}(\dot{B}^{-\theta+\delta_2}_{2,\infty})_{\rm v}}
	\Vert b \Vert_{(\dot{B}^{-3\alpha(r)+\theta-\delta_1}_{2,1})_{\rm h}(\dot{B}^{-\theta-\delta_2}_{2,1})_{\rm v}}
\end{aligned}
$$
For simplicity, we denote $(\ \cdot\ |\ \cdot\ )$ as the inner product in $L^2$.

\noindent$\bullet$ Estimate of $\left(Q_1(v,v)|\partial_3v^3\right)_{\mathcal{H}^{\theta,r}}$.\\
According to Lemma \ref{Sobolev type} and Lemma \ref{anisotropic product}, taking $r$ sufficiently close to $2$ we have
\begin{equation}\label{Q1 estimate}
	\begin{aligned}
		&\left(Q_1(v,v)|\partial_3v^3\right)_{\mathcal{H}^{\theta,r}}\\
		\lesssim& \Vert\partial_3v^3\Vert_{(\dot{B}^{\frac 2 p-3\alpha(r)+\theta}_{2,\infty})_{\rm h}(\dot{B}^{-\theta -\frac 1 2+3\alpha(r)}_{2,\infty})_{\rm v}} \Vert(\omega+\partial_3v^3)(\omega+\partial_3v^3)\Vert_{(\dot{B}^{-\frac 2 p-3\alpha(r)+\theta}_{2,1})_{\rm h}(\dot{B}^{\frac 1 2-\theta-3\alpha(r)}_{2,1})_{\rm v}}\\
		\lesssim&\Vert v^3\Vert_{\dot{B}^{\frac 1 2+\frac 2 p}_{2,\infty}} \left(\Vert\omega\Vert_{\dot{H}^{\frac 1 2(1-3\alpha(r)+\theta-\frac 2 p),\frac 1 2(1-\theta-3\alpha(r))}}+\Vert\partial_3v^3\Vert_{\dot{H}^{-3\alpha(r)+\theta+\frac 1 2(1+3\alpha(r)-\theta-\frac 2 p),-\theta+\frac 1 2(1+\theta-3\alpha(r))}}\right)^2\\
		\lesssim&\Vert v^3\Vert_{\dot{B}^{\frac 1 2+\frac 2 p}_{2,\infty}} \left(\Vert\omega\Vert_{\dot{H}^{1-\frac 1 p-3\alpha(r)}}+\Vert|\nabla|^{1-\frac 1 p}\partial_3v^3\Vert_{\mathcal{H}^{\theta,r}}\right)^2\\
		\lesssim&\Vert v^3\Vert_{\dot{B}^{\frac 1 2+\frac 2 p}_{2,\infty}} \left(\Vert\omega_{\frac r 2}\Vert_2^{2(2\alpha(r)+\frac 1 p)}\Vert\nabla\omega_{\frac r 2}\Vert_2^{2-\frac 2 p}+\Vert\partial_3v^3\Vert^{\frac 2 p}_{\mathcal{H}^{\theta,r}}\Vert\nabla\partial_3v^3\Vert^{2-\frac 2 p}_{\mathcal{H}^{\theta,r}}\right).\\
	\end{aligned}
\end{equation}

\noindent$\bullet$ Estimate of $\left(Q_2(v,v)|\partial_3v^3\right)_{\mathcal{H}^{\theta,r}}$.\\
\quad The estimate of  $\left(Q_2(v,v)|\partial_3v^3\right)_{\mathcal{H}^{\theta,r}}$ can de reduced to
$$\left(\partial_3 v^l\cdot\partial_lv^3|\partial_3 v^3\right)_{\mathcal{H}^{\theta,r}}\ {\rm with}\ l=1,2.$$
According to Lemma \ref{anisotropic product}, for $p>4$ we have
$$
\begin{aligned}
	&\left|\left(\partial_3 v^l\cdot\partial_lv^3|\partial_3 v^3\right)_{\mathcal{H}^{\theta,r}}\right|\\
	\lesssim&\Vert\partial_3 v^3\Vert_{\dot{H}^{-3\alpha(r)+\theta+(1-3\alpha(r)-\frac 2 p),-\theta+(3\alpha(r)+\frac 2 p)}}\Vert\partial_l v^3\partial_3 v^l\Vert_{\dot{H}^{\theta-1+\frac 2 p,-\theta-3\alpha(r)-\frac 2 p}}\\
	\lesssim&\Vert\nabla\partial_3v^3\Vert_{\mathcal{H}^{\theta,r}}\Vert\partial_lv^3\Vert_{(\dot{B}^{\frac 1 2-\theta}_{2,\infty})_{\rm v}(\dot{B}^{\theta-1+\frac 2 p}_{2,2})_{\rm h}}\Vert\partial_3 v^l\Vert_{(\dot{B}^{-3\alpha(r)-\frac 2 p}_{2,1})_{\rm v}(\dot{B}^{1}_{2,1})_{\rm h}}\\
	\lesssim&\Vert v^3\Vert_{\dot{B}^{\frac 1 2+\frac 2 p}_{2,\infty}} \Vert\nabla\partial_3v^3\Vert_{\mathcal{H}^{\theta,r}} \Vert v^l\Vert_{(\dot{B}^{1}_{2,1})_{\rm h}(\dot{B}^{1-3\alpha(r)-\frac 2 p}_{2,1})_{\rm v}}.
\end{aligned}
$$
Thanks to Proposition \ref{v^h} we get that
\begin{equation}\label{Q2 estimate}
	\left(Q_2(v,v)|\partial_3v^3\right)_{\mathcal{H}^{\theta,r}}\lesssim\Vert v^3\Vert_{\dot{B}^{\frac 1 2+\frac 2 p}_{2,\infty}} \Vert\nabla\partial_3v^3\Vert_{\mathcal{H}^{\theta,r}}\left(
	\Vert\omega_{\frac r 2}\Vert_2^{2\alpha(r)+\frac 2 p}\Vert\nabla\omega_{\frac r 2}\Vert_2^{1-\frac 2 p}+\Vert\partial_3v^3\Vert_{\mathcal{H}^{\theta,r}}^{\frac 2 p}\Vert\nabla\partial_3v^3\Vert_{\mathcal{H}^{\theta,r}}^{1-\frac 2 p}
	\right).
\end{equation}

\noindent$\bullet$ Estimate of $\left(Q_3(v,v)|\partial_3v^3\right)_{\mathcal{H}^{\theta,r}}$.\\
\quad We start with the following lemma:
\begin{lemm}\label{v^h}
	The following inequality holds true.
	$$
	\begin{aligned}
		&\left|\left(v^{\rm h}\cdot\nabla_{\rm h}\partial_3 v^3|\partial_3 v^3\right)_{\mathcal{H}^{\theta,r}}\right|\lesssim
		\Vert v^3\Vert_{\dot{B}^{\frac 1 2+\frac 2 p}_{2,\infty}}
		\left(\left\Vert\nabla_{\rm h}v^{\rm h}\right\Vert_{\dot{H}^{\frac 1 2-3\alpha(r)+\theta,\frac 1 2-\frac 1 p-\theta}}^2\right.\\
		&\quad\quad\left.+\left\Vert|\nabla|^{1-\frac 1 p}\partial_3 v^3\right\Vert_{\mathcal{H}^{\theta,r}}^2
		+\Vert\nabla\partial_3v^3\Vert_{\mathcal{H}^{\theta,r}}\left\Vert v^{\rm h}\right\Vert_{\left(\dot{B}^{1}_{2,1}\right)_{\rm h}(\dot{B}^{1-3\alpha(r)-\frac 2 p}_{2,1})_{\rm v}}\right).
	\end{aligned}
	$$
\begin{proof}
	Since the space $\mathcal{H}^{\theta,r}$ is equivalent to the anisotropic Besov space $(\dot{B}^{-3\alpha(r)+\theta}_{2,2})_{\rm h}(\dot{B}^{-\theta}_{2,2})_{\rm v}$, it follows that
	\begin{equation}\label{inner of H}
		\left(v^{\rm h}\cdot\nabla_{\rm h}\partial_3v^3|\partial_3v^3\right)_{\mathcal{H}^{\theta, r}}=\sum_{k,l\in\mathbb{Z}}2^{2k(-3\alpha(r)+\theta)}2^{-2l\theta}\left(\Delta_k^{\rm h}\Delta_l^{\rm v}(v^{\rm h}\cdot\nabla_{\rm h}\partial_3v^3)|\Delta_k^{\rm h}\Delta_l^{\rm v}\partial_3v^3\right)
	\end{equation}
	Applying Bony's decomposition(see Proposition \ref{isotropic Bony}) to $v^{\rm h}\cdot\nabla_{\rm h}\partial_3v^3$ for both horizontal and vertical variables, we write that
	\begin{equation}\label{Bony}
		\begin{aligned}
			&v^{\rm h}\cdot\nabla_{\rm h}\partial_3v^3=(T^{\rm h}+R^{\rm h}+\tilde{T}^{\rm h})(T^{\rm v}+R^{\rm v}+\tilde{T}^{\rm v})(v^{\rm h},\nabla_{\rm h}\partial_3 v^3)=T^{\rm h}T^{\rm v}(v^{\rm h},\ \nabla_{\rm h}\partial_3 v^3)+A+B\ {\rm with}\\
			&A\triangleq T^{\rm h}R^{\rm v}(v^{\rm h},\ \nabla_{\rm h}\partial_3 v^3)+T^{\rm h}\tilde{T}^{\rm v}(v^{\rm h},\nabla_{\rm h}\partial_3 v^3)\ {\rm and}\ B\triangleq (R^{\rm h}+\tilde{T}^{\rm h})(T^{\rm v}+R^{\rm v}+\tilde{T}^{\rm v})(v^{\rm h},\ \nabla_{\rm h}\partial_3 v^3).
		\end{aligned}		
	\end{equation}
	$\bullet$ Estimate $\left(\Delta_k^{\rm h}\Delta_l^{\rm v}T^{\rm h}T^{\rm v}(v^{\rm h},\ \nabla_{\rm h}\partial_3v^3)|\Delta_k^{\rm h}\Delta_l^{\rm v}\partial_3v^3\right)$.\\
By virtue of the properties of the support to the Fourier transform, we write that
	$$
	\begin{aligned}
		&\left(\Delta_k^{\rm h}\Delta_l^{\rm v}T^{\rm h}T^{\rm v}(v^{\rm h},\ \nabla_{\rm h}\partial_3v^3)|\Delta_k^{\rm h}\Delta_l^{\rm v}\partial_3v^3\right)\triangleq I_{k,l}^1+I_{k,l}^2+I_{k,l}^3\ {\rm with}\\
		&I_{k,l}^1\triangleq\sum_{\substack{|k'-k|\leq 4\\ |l'-l|\leq 4}}\left(\left[\Delta_k^{\rm h}\Delta_l^{\rm v},\ S^{\rm h}_{k'-1} S^{\rm v}_{l'-1}v^{\rm h}\right]\Delta_{k'}^{\rm h}\Delta_{l'}^{\rm v}\nabla_{\rm h}\partial_3v^3|\Delta_k^{\rm h}\Delta_l^{\rm v}\partial_3v^3\right),\\
		&I_{k,l}^2\triangleq \sum_{\substack{|k'-k|\leq 4\\ |l'-l|\leq 4}}\left(\left(S_{k'-1}^{\rm h}S_{l'-1}^{\rm v}v^{\rm h}-S_{k-1}^{\rm h}S_{l-1}^{\rm v}v^{\rm h}\right)\Delta_{k'}^{\rm h}\Delta_{l'}^{\rm v}\Delta_{k}^{\rm h}\Delta_{l}^{\rm v}\nabla_{\rm h}\partial_3v^3|\Delta_k^{\rm h}\Delta_l^{\rm v}\partial_3v^3\right),\\
		&I_{k,l}^3\triangleq-\frac 1 2\left(S_{k-1}^{\rm h}S_{l-1}^{\rm v}{\rm div}_{\rm h}v^{\rm h}\Delta_{k}^{\rm h}\Delta_{l}^{\rm v}\partial_3v^3|\Delta_k^{\rm h}\Delta_l^{\rm v}\partial_3v^3\right).
	\end{aligned}
	$$
	The standard commutator estimate (see \cite{BCD}) ensures that
	$$
		|I_{k,l}^1|\lesssim \sum_{\substack{|k'-k|\leq 4\\ |l'-l|\leq 4}}\left(2^{-k}\left\Vert S_{k'-1}^{\rm h}S_{l'-1}^{\rm v}\nabla_{\rm h}v^{\rm h}\right\Vert_\infty+2^{-l}\left\Vert S_{k'-1}^{\rm h}S_{l'-1}^{\rm v}\partial_3v^{\rm h}\right\Vert_\infty\right)\Vert\Delta_{k'}^{\rm h}\Delta_{l'}^{\rm v}\nabla_{\rm h}\partial_3v^3\Vert_2 \Vert\Delta_k^{\rm h}\Delta_l^{\rm v}\partial_3v^3\Vert_2.
	$$
	Using Lemma \ref{Bernstein}, we get
	$$
	\left\Vert S_{k'-1}^{\rm h}S_{l'-1}^{\rm v}\nabla_{\rm h}v^{\rm h}\right\Vert_\infty\lesssim2^{k'(\frac 1 2+3\alpha(r)-\theta)}2^{l'(\frac 1 p+\theta)} c_{k',l'}\Vert\nabla_{\rm h}v^{\rm h}\Vert_{\dot{H}^{\frac 1 2-3\alpha(r)+\theta,\frac 1 2-\frac 1 p+\theta}},
	$$
	$$
	\left\Vert S_{k'-1}^{\rm h}S_{l'-1}^{\rm v}\partial_3v^{\rm h}\right\Vert_\infty\lesssim2^{l'(\frac 1 2+3\alpha(r)+\frac 2 p)}d_{l'}\Vert v^{h}\Vert_{(\dot{B}_{2,1}^{1})_{\rm h}(\dot{B}_{2,1}^{1-3\alpha(r)-\frac 2 p})_{\rm v}}.
	$$
	Here and in what follows, we always denote $\left(c_{k,l}\right)_{k,l\in\mathbb{Z}^2}$(resp. $\left(d_{k,l}\right)_{k,l\in\mathbb{Z}^2}$) a generic element of the sphere in $l^2(\mathbb{Z}^2)$(resp. $l^1(\mathbb{Z}^2)$), and $\left(c_{k}\right)_{k\in\mathbb{Z}}$(resp. $\left(d_{k}\right)_{k\in\mathbb{Z}}$) a generic element of the sphere in $l^2(\mathbb{Z})$(resp. $l^1(\mathbb{Z})$).\\
From the above inequalities we infer that
	$$
	\begin{aligned}
		&|I_{k,l}^1|\\
		\lesssim& \sum_{\substack{|k'-k|\leq 4\\ |l'-l|\leq 4}}c_{k',l'}2^{2k'(3\alpha(r)-\theta)}2^{l'(-\frac 1 2+\frac 2 p+2\theta)}\Vert\nabla_{\rm h}v^{\rm h}\Vert_{\dot{H}^{\frac 1 2-3\alpha(r)+\theta,\frac 1 2-\frac 1 p+\theta}}\cdot 2^{l(\frac 1 2-\frac 2 p +\varepsilon)}2^{-k\varepsilon}\Vert\partial_3v^3\Vert_{(\dot{B}^{-\frac 1 2+\frac 2 p-\varepsilon}_{2,\infty})_{\rm v}(\dot{B}^{\varepsilon}_{2,2})_{\rm h}} \\
		&\quad \quad \quad \quad \cdot c_{k',l'}2^{k'\varepsilon} 2^{-l'\varepsilon}\Vert \partial_3 v^3\Vert_{\dot{H}^{\frac 1 2-3\alpha(r)+\theta-\varepsilon,\frac 1 2-\frac 1 p+\theta+\varepsilon}}\\
		&+\sum_{\substack{|k'-k|\leq 4\\ |l'-l|\leq 4}}2^{-l}2^{k'(1-\frac 2 p)}2^{l'(1+3\alpha(r)+\frac 2 p)}d_{l'}\Vert v^{h}\Vert_{(\dot{B}_{2,1}^{1})_{\rm h}(\dot{B}_{2,1}^{1-3\alpha(r)-\frac 2 p})_{\rm v}}c_{k'}\Vert\nabla_{\rm h}v^3\Vert_{(\dot{B}^{\frac 1 2}_{2,\infty})_{\rm v}(\dot{B}^{-1+\frac 2 p}_{2,2})_{\rm h}}\\
		&\quad \quad \quad \quad \cdot c_{k,l}2^{-k(1-\frac 2 p-6\alpha(r)+2\theta)}2^{-l(\frac 2 p+3\alpha(r)-2\theta)}\Vert\partial_3v^3\Vert_{\dot{H}^{1-\frac 2 p-6\alpha(r)+2\theta,\frac 2 p+3\alpha(r)-2\theta}},
	\end{aligned}
	$$
	applying the H\"{o}lder inequality and Lemma \ref{embed2}, we have
	\begin{equation}
		\begin{aligned}
			|I_{k,l}^1|
			\lesssim&
			d_{k,l}2^{2k(3\alpha(r)-\theta)}2^{2l\theta}\Vert v^3\Vert_{\dot{B}^{\frac 1 2+\frac 2 p}_{2,\infty}}
			\left(\left\Vert\nabla_{\rm h}v^{\rm h}\right\Vert_{\dot{H}^{\frac 1 2-3\alpha(r)+\theta,\frac 1 2-\frac 1 p-\theta}}^2\right.\\
			&+\left.\left\Vert|\nabla|^{1-\frac 1 p}\partial_3 v^3\right\Vert_{\mathcal{H}^{\theta,r}}^2
			+\Vert\nabla\partial_3v^3\Vert_{\mathcal{H}^{\theta,r}}\left\Vert v^{\rm h}\right\Vert_{\left(\dot{B}^{1}_{2,1}\right)_{\rm h}(\dot{B}^{1-3\alpha(r)-\frac 2 p}_{2,1})_{\rm v}}\right).
		\end{aligned}		
	\end{equation}
	By the similar argument we can obtain the same estimate for $I_{k,l}^2$ and $I_{k,l}^3$, we thus conclude that

	\begin{equation}\label{TT estimate}
		\begin{aligned}
			|\left(\Delta_k^{\rm h}\Delta_l^{\rm v}T^{\rm h}T^{\rm v}(v^{\rm h},\ \nabla_{\rm h}\partial_3v^3)|\Delta_k^{\rm h}\Delta_l^{\rm v}\partial_3v^3\right)|
			\lesssim&
			d_{k,l}2^{2k(3\alpha(r)-\theta)}2^{2l\theta}\Vert v^3\Vert_{\dot{B}^{\frac 1 2+\frac 2 p}_{2,\infty}}
			\left(\left\Vert\nabla_{\rm h}v^{\rm h}\right\Vert_{\dot{H}^{\frac 1 2-3\alpha(r)+\theta,\frac 1 2-\frac 1 p-\theta}}^2\right.\\
			&+\left.\left\Vert|\nabla|^{1-\frac 1 p}\partial_3 v^3\right\Vert_{\mathcal{H}^{\theta,r}}^2
			+\Vert\nabla\partial_3v^3\Vert_{\mathcal{H}^{\theta,r}}\left\Vert v^{\rm h}\right\Vert_{\left(\dot{B}^{1}_{2,1}\right)_{\rm h}(\dot{B}^{1-3\alpha(r)-\frac 2 p}_{2,1})_{\rm v}}\right).
		\end{aligned}		
	\end{equation}
	
	\noindent$\bullet$ Estimate of $\left(\Delta_k^{\rm h}\Delta_l^{\rm v}A|\Delta_k^{\rm h}\Delta_l^{\rm v}\partial_3v^3\right)$.\\
	Applying Lemma \ref{Bernstein}, we have that
	\begin{equation}\label{terms in A}
		\begin{gathered}
			\Vert S_{k'-1}^{\rm h}\Delta_{l'}^{\rm v}v^{\rm h}\Vert_{L^\infty_{\rm h}L^2_{\rm v}}\lesssim c_{l'}2^{-l'(1-3\alpha(r)-\frac 2 p)}\Vert v^{\rm h}\Vert_{(\dot{B}_{2,1}^{1})_{\rm h}(\dot{B}_{2,1}^{1-3\alpha(r)-\frac 2 p})_{\rm v}},\\
			\Vert\Delta_{k'}^{\rm h}S_{l'-1}^{\rm v}\nabla_{\rm h}\partial_3 v^3\Vert_{L^2_{\rm h}L^\infty_{\rm v}}\lesssim c_{k'} 2^{k'(1-\frac 2 p)}2^{l'}\Vert v^3\Vert_{(\dot{B}^{\frac 1 2}_{2,\infty})_{\rm v}(\dot{B}^{\frac 2 p}_{2,2})_{\rm h}}.
		\end{gathered}
	\end{equation}
	Also considering the support to the Fourier transform to the terms in $T^{\rm h}R^{\rm v}(v^{\rm h},\nabla_{\rm h}\partial_3v^3)$, we have
	$$
	\begin{aligned}
		&\Vert\Delta_k^{\rm h}\Delta_l^{\rm v}T^{\rm h}R^{\rm v}(v^{\rm h},\ \nabla_{\rm h}\partial_3 v^3)\Vert_2\\
		\lesssim&2^{\frac l 2}\sum_{\substack{|k'-k|\leq 4\\ l'\geq l-3}}\Vert S_{k'-1}^{\rm h}\Delta_{l'}^{\rm v}v^{\rm h}\Vert_{L^\infty_{\rm h}L^2_{\rm v}}\Vert\Delta_{k'}^{\rm h}\tilde{\Delta}_{l'}^{\rm v}\nabla_{\rm h}\partial_3 v^3\Vert_2\\
		\lesssim&2^{\frac l 2}\sum_{\substack{|k'-k|\leq 4\\ l'\geq l-3}}c_{k'}d_{l'}2^{2k'(\frac 1 p+3\alpha(r)-\theta)}2^{-l'(1-2\theta)}\left\Vert v^{\rm h}\right\Vert_{\left(\dot{B}^{1}_{2,1}\right)_{\rm h}(\dot{B}^{1-3\alpha(r)-\frac 2 p}_{2,1})_{\rm v}}\Vert\partial_2 v^3\Vert_{\dot{H}^{1-\frac 2 p-6\alpha(r)+2\theta,\frac 2 p+3\alpha(r)-2\theta}}\\
		\lesssim& c_k\cdot d_l\cdot2^{2k(\frac 1 p+3\alpha(r)-\theta)}2^{-l(\frac 1 2-2\theta)}\Vert\nabla\partial_3v^3\Vert_{\mathcal{H}^{\theta,r}}\left\Vert v^{\rm h}\right\Vert_{\left(\dot{B}^{1}_{2,1}\right)_{\rm h}(\dot{B}^{1-3\alpha(r)-\frac 2 p}_{2,1})_{\rm v}}.
	\end{aligned}
	$$
	Therefore according to Lemma \ref{embed2}, we deduce that
	$$\Vert\Delta_{k}^{\rm h}\Delta_l^{\rm v}\partial_3v^3\Vert_2\lesssim c_k\cdot 2^{-k\frac 2 p}2^{\frac l 2}\Vert v^3\Vert_{(\dot{B}^{\frac 1 2}_{2,\infty})_{\rm v}(\dot{H}^{\frac 2 p})_{\rm h}}\lesssim c_k\cdot 2^{-k\frac 2 p}2^{\frac l 2}\Vert v^3\Vert_{\dot{B}^{\frac 1 2+\frac 2 p}_{2,\infty}},
	$$
	we obtain
	\begin{equation}\label{TR estimate}
		\begin{aligned}
			\left|\left(\Delta_k^{\rm h}\Delta_l^{\rm v}T^{\rm h}R^{\rm v}(v^{\rm h},\ \nabla_{\rm h}\partial_3 v^3)|\Delta_{k}^{\rm h}\Delta_l^{\rm v}\partial_3v^3\right)\right|\lesssim& d_{k,l}2^{2k(3\alpha(r)-\theta)}2^{2l\theta}\Vert v^3\Vert_{\dot{B}^{\frac 1 2+\frac 2 p}_{2,\infty}}\\
			&\times\Vert\nabla\partial_3v^3\Vert_{\mathcal{H}^{\theta,r}}\left\Vert v^{\rm h}\right\Vert_{\left(\dot{B}^{1}_{2,1}\right)_{\rm h}(\dot{B}^{1-3\alpha(r)-\frac 2 p}_{2,1})_{\rm v}}.
		\end{aligned}
	\end{equation}
Along the same lines, we infer from \eqref{terms in A} that
	\begin{equation}\label{TT-1 estimate}
		\begin{aligned}
			&\left|\left(\Delta_k^{\rm h}\Delta_l^{\rm v}T^{\rm h}\tilde{T}^{\rm v}(v^{\rm h},\ \nabla_{\rm h}\partial_3 v^3)|\Delta_{k}^{\rm h}\Delta_l^{\rm v}\partial_3v^3\right)\right|\\
			\lesssim & \sum_{\substack{|k'-k|\leq 4\\ |l'-l|\leq 4}}\Vert S_{k'-1}^{\rm h}\Delta_{l'}^{\rm v}v^{\rm h}\Vert_{L^\infty_{\rm h}L^2_{\rm v}}\Vert\Delta_{k'}^{\rm h}S_{l'-1}^{\rm v}\nabla_{\rm h}\partial_3 v^3\Vert_{L^2_{\rm h}L^\infty_{\rm v}}\Vert\Delta_{k}^{\rm h}\Delta_l^{\rm v}\partial_3v^3\Vert_2\\
			\lesssim& \sum_{\substack{|k'-k|\leq 4\\ |l'-l|\leq 4}} c_{l'}2^{-l'(1-3\alpha(r)-\frac 2 p)}\Vert v^{\rm h}\Vert_{(\dot{B}_{2,1}^{1})_{\rm h}(\dot{B}_{2,1}^{1-3\alpha(r)-\frac 2 p})_{\rm v}}c_{k'} 2^{k'(1-\frac 2 p)}2^{l'}\Vert v^3\Vert_{(\dot{B}^{\frac 1 2}_{2,\infty})_{\rm v}(\dot{B}^{\frac 2 p}_{2,2})_{\rm h}}\\
			&\times c_{k,l}2^{k(-1+\frac 2 p+6\alpha(r)-2\theta)}2^{l(-\frac 2 p-3\alpha(r)+2\theta)}\Vert\partial_3 v^3\Vert_{\dot{H}^{1-\frac 2 p-6\alpha(r)+2\theta,\frac 2 p+3\alpha(r)-2\theta}}\\
			\lesssim&d_{k,l}\cdot2^{2k(3\alpha(r)-\theta)}2^{2l\theta}\Vert v^3\Vert_{\dot{B}^{\frac 1 2+\frac 2 p}_{2,\infty}}\Vert\nabla\partial_3v^3\Vert_{\mathcal{H}^{\theta,r}}\left\Vert v^{\rm h}\right\Vert_{\left(\dot{B}^{1}_{2,1}\right)_{\rm h}(\dot{B}^{1-3\alpha(r)-\frac 2 p}_{2,1})_{\rm v}},
		\end{aligned}
	\end{equation}
	from \eqref{TR estimate} and \eqref{TT-1 estimate}, we conclude that
	\begin{equation}\label{A estimate}
		\left|\left(\Delta_k^{\rm h}\Delta_l^{\rm v}A|\Delta_{k}^{\rm h}\Delta_l^{\rm v}\partial_3v^3\right)\right|\lesssim d_{k,l}\cdot2^{2k(3\alpha(r)-\theta)}2^{2l\theta}\Vert v^3\Vert_{\dot{B}^{\frac 1 2+\frac 2 p}_{2,\infty}}\Vert\nabla\partial_3v^3\Vert_{\mathcal{H}^{\theta,r}}\left\Vert v^{\rm h}\right\Vert_{\left(\dot{B}^{1}_{2,1}\right)_{\rm h}(\dot{B}^{1-3\alpha(r)-\frac 2 p}_{2,1})_{\rm v}}.
	\end{equation}
	
	\noindent$\bullet$ Estimate of $\left(\Delta_k^{\rm h}\Delta_l^{\rm v}B|\Delta_k^{\rm h}\Delta_l^{\rm v}\partial_3v^3\right)$.\\
Notice that the support of the  Fourier transform to terms in $R^{\rm h}R^{\rm v}(v^{\rm h},\ \nabla_{\rm h}\partial_3v^3)$. Since $\theta\in]0,\ \frac 1 2-\frac 1 p[$, we can choose $\varepsilon>0$ small enough such that $2\theta-6\alpha(r)+1-\varepsilon>0,\ 1-\frac 2 p-2\theta+\varepsilon>0$, and $\frac 1 2+\frac 2 p-\varepsilon>0$, then apply the anisotropic Bernstein inequality to obtain
	$$
	\begin{aligned}
		&\Vert\Delta_k^{\rm h}\Delta_l^{\rm v}R^{\rm h}R^{\rm v}(v^{\rm h},\ \nabla_{\rm h}\partial_3 v^3)\Vert_2\\
		\lesssim&2^k2^{\frac l 2}\sum_{\substack{k'\geq k-3\\ l'\geq l-3}}\Vert\Delta_{k'}^{\rm h}\Delta_{l'}^{\rm v}v^{\rm h}\Vert_2 \Vert\tilde{\Delta}_{k'}^{\rm h}\tilde{\Delta}_{l'}^{\rm v}\nabla_{\rm h}\partial_3v^3\Vert_2\\
		\lesssim& 2^{k}2^{\frac l 2}\sum_{\substack{k'\geq k-3\\ l'\geq l-3}}d_{k',l'}2^{-k'(1-6\alpha(r)+2\theta-\varepsilon)}2^{-l'(1-\frac 2 p-2\theta+\varepsilon)}\Vert\nabla_{\rm h}v^{\rm h}\Vert_{\dot{H}^{\frac 1 2-3\alpha(r)+\theta,\frac 1 2-\frac 1 p-\theta}}\Vert\partial_3 v^3\Vert_{\dot{H}^{\frac 1 2-3\alpha(r)+\theta-\varepsilon,\frac 1 2-\frac 1 p-\theta+\varepsilon}}\\
		\lesssim& d_{k,l}\cdot2^{2k(3\alpha(r)-\theta)}2^{-l(\frac 1 2-\frac 2 p-2\theta)}2^{k\varepsilon}2^{-l\varepsilon}\Vert\nabla_{\rm h}v^h\Vert_{\dot{H}^{\frac 1 2-3\alpha(r)+\theta,\frac 1 2-\frac 1 p-\theta}}\Vert|\nabla|^{1-\frac 1 p}\partial_3v^3\Vert_{\mathcal{H}^{\theta,r}},
	\end{aligned}
	$$
	$$
	\begin{aligned}
		&\Vert\Delta_k^{\rm h}\Delta_l^{\rm v}R^{\rm h}T^{\rm v}(v^{\rm h},\ \nabla_{\rm h}\partial_3 v^3)\Vert_2\\
		\lesssim&2^k\sum_{\substack{k'\geq k-3\\ |l'-l|\leq 4}}\Vert\Delta_{k'}^{\rm h}S_{l'-1}^{\rm v}v^{\rm h}\Vert_{L^2_{\rm h}L^\infty_{\rm v}} \Vert\tilde{\Delta}_{k'}^{\rm h}\Delta_{l'}^{\rm v}\nabla_{\rm h}\partial_3v^3\Vert_2\\
		\lesssim& 2^{k}\sum_{\substack{k'\geq k-3\\ |l'-l|\leq 4}}d_{k',l'}2^{-k'(1-6\alpha(r)+2\theta-\varepsilon)}2^{-l'(1-\frac 2 p-2\theta+\varepsilon)}\Vert\nabla_{\rm h}v^{\rm h}\Vert_{\dot{H}^{\frac 1 2-3\alpha(r)+\theta,\frac 1 2-\frac 1 p-\theta}}\Vert\partial_3 v^3\Vert_{\dot{H}^{\frac 1 2-3\alpha(r)+\theta-\varepsilon,\frac 1 2-\frac 1 p-\theta+\varepsilon}}\\
		\lesssim& d_{k,l}\cdot2^{2k(3\alpha(r)-\theta)}2^{-l(\frac 1 2-\frac 2 p-2\theta)}2^{k\varepsilon}2^{-l\varepsilon}\Vert\nabla_{\rm h}v^h\Vert_{\dot{H}^{\frac 1 2-3\alpha(r)+\theta,\frac 1 2-\frac 1 p-\theta}}\Vert|\nabla|^{1-\frac 1 p}\partial_3v^3\Vert_{\mathcal{H}^{\theta,r}},\\
		&\Vert\Delta_{k}^{\rm h}\Delta_{l}^{\rm v}\partial_3 v^3\Vert_2\lesssim 2^{-k\varepsilon}2^{l(\frac 1 2-\frac 2 p+\varepsilon)}\Vert\partial_3v^3\Vert_{(\dot{B}^{\varepsilon}_{2,\infty})_{\rm h}(\dot{B}^{-\frac 1 2+\frac 2 p-\varepsilon}_{2,\infty})_{\rm v}}\lesssim 2^{-k\varepsilon}2^{l(\frac 1 2-\frac 2 p+\varepsilon)}\Vert v^3\Vert_{\dot{B}^{\frac 1 2+\frac 2 p}_{2,\infty}}.
	\end{aligned}
	$$
	The remaining terms in $B$ can be estimated by the same way. Therefore we obtain by Young's inequality
	\begin{equation}\label{B estimate}
		\begin{aligned}
			&\left|\left(\Delta_k^{\rm h}\Delta_l^{\rm v}B|\Delta_{k}^{\rm h}\Delta_l^{\rm v}\partial_3v^3\right)\right|\\
			&\lesssim d_{k,l}\cdot2^{2k(3\alpha(r)-\theta)}2^{2l\theta}\Vert v^3\Vert_{\dot{B}^{\frac 1 2+\frac 2 p}_{2,\infty}}\left(\left\Vert\nabla_{\rm h}v^{\rm h}\right\Vert_{\dot{H}^{\frac 1 2-3\alpha(r)+\theta,\frac 1 2-\frac 1 p-\theta}}^2.+\left\Vert|\nabla|^{1-\frac 1 p}\partial_3 v^3\right\Vert_{\mathcal{H}^{\theta,r}}^2\right).
		\end{aligned}
	\end{equation}
	Combining \eqref{inner of H}, \eqref{TT estimate}, \eqref{A estimate} and \eqref{B estimate}, we conclude the result of Lemma \ref{v^h}.
\end{proof}
\end{lemm}

Thanks to Lemma \ref{v^h}, applying H\"{o}lder's inequality of anisotropic Sobolev spaces and Proposition \ref{anisoropic v^h}, we infer that
\begin{equation}\label{v^h in flow term}
	\begin{aligned}
		&\left|\left(v^{\rm h}\cdot\nabla_{\rm h}\partial_3 v^3|\partial_3 v^3\right)_{\mathcal{H}^{\theta,r}}\right|\lesssim
		\Vert v^3\Vert_{\dot{B}^{\frac 1 2+\frac 2 p}_{2,\infty}}
		\left(\left\Vert\omega_{\frac r 2}\right\Vert_2
		^{2(2\alpha(r)+\frac 1 p)}\left\Vert\nabla\omega_{\frac r 2}\right\Vert_2^{2-\frac 2 p}\right.\\
		&\quad\quad\left.+\left(
		\left\Vert\omega_{\frac r 2}\right\Vert_2
		^{2(\alpha(r)+\frac 1 p)}\left\Vert\nabla\omega_{\frac r 2}\right\Vert_2^{1-\frac 2 p}
		+\Vert\partial_3v^3\Vert_{\mathcal{H}^{\theta,r}}^{\frac 2 p}\Vert\nabla\partial_3v^3\Vert_{\mathcal{H}^{\theta,r}}^{1-\frac 2 p}
		\right)\Vert\nabla\partial_3v^3\Vert_{\mathcal{H}^{\theta,r}}\right).
	\end{aligned}
\end{equation}
For $\varepsilon>0$ small enough, we have
\begin{equation}\label{v^3 in flow term}
	\begin{aligned}
		\left|\left(v^3\cdot\partial_3^2 v^3|\partial_3 v^3\right)_{\mathcal{H}^{\theta,r}}\right|
		\lesssim&\Vert\partial_3v^3\Vert_{\dot{H}^{-3\alpha(r)+\theta+1-\frac 2 p-\varepsilon,-\theta+\varepsilon}} \Vert v^3\cdot\partial_3^2v^3\Vert_{\dot{H}^{-3\alpha(r)+\theta-1+\frac 2 p+\varepsilon,-\theta-\varepsilon}}\\
		\lesssim&\Vert|\nabla|^{1-\frac 2 p}\partial_3v^3\Vert_{\mathcal{H}^{\theta,r}}\Vert v^3\Vert_{(\dot{B}^{\frac 2 p+\varepsilon}_{2,\infty})_{\rm h} (\dot{B}^{\frac 1 2-\varepsilon}_{2,\infty})_{\rm v}}\Vert\partial_3^2v^3\Vert_{\mathcal{H}^{\theta,r}}\\
		\lesssim&\Vert v^3\Vert_{\dot{B}^{\frac 1 2+\frac 2 p}_{2,\infty}}\Vert\partial_3v^3\Vert^{\frac 2 p}_{\mathcal{H}^{\theta,r}}\Vert\nabla\partial_3v^3\Vert^{2-\frac 2 p}_{\mathcal{H}^{\theta,r}}.
	\end{aligned}
\end{equation}
Thus combining with \eqref{v^h in flow term} and \eqref{v^3 in flow term}, we get that
\begin{equation}\label{Q3 estimate}
	\begin{aligned}
		&\left|\left(Q_3(v,v)|\partial_3 v^3\right)_{\mathcal{H}^{\theta,r}}\right|\lesssim
		\Vert v^3\Vert_{\dot{B}^{\frac 1 2+\frac 2 p}_{2,\infty}}
		\left(\left\Vert\omega_{\frac r 2}\right\Vert_2
		^{2(2\alpha(r)+\frac 1 p)}\left\Vert\nabla\omega_{\frac r 2}\right\Vert_2^{2-\frac 2 p}\right.\\
		&\quad\quad\left.+\left(
		\left\Vert\omega_{\frac r 2}\right\Vert_2
		^{2(\alpha(r)+\frac 1 p)}\left\Vert\nabla\omega_{\frac r 2}\right\Vert_2^{1-\frac 2 p}
		+\Vert\partial_3v^3\Vert_{\mathcal{H}^{\theta,r}}^{\frac 2 p}\Vert\nabla\partial_3v^3\Vert_{\mathcal{H}^{\theta,r}}^{1-\frac 2 p}
		\right)\Vert\nabla\partial_3v^3\Vert_{\mathcal{H}^{\theta,r}}\right).
	\end{aligned}
\end{equation}
Combining \eqref{Q1 estimate}, \eqref{Q2 estimate} and \eqref{Q3 estimate}, we obtain
\begin{equation}\label{v^3 estimate}
	\begin{aligned}
		&\frac{\rm d}{{\rm d}t}\Vert\partial_3v^3\Vert_{\mathcal{H}^{\theta,r}}^2+\Vert\nabla\partial_3v^3\Vert_{\mathcal{H}^{\theta,r}}^2\\
		\lesssim&\Vert v^3\Vert_{\dot{B}^{\frac 1 2+\frac 2 p}_{2,\infty}} \left\Vert\omega_{\frac r 2}\right\Vert_2
		^{2(2\alpha(r)+\frac 1 p)}\left\Vert\nabla\omega_{\frac r 2}\right\Vert_2^{2-\frac 2 p}\\
		&+\Vert v^3\Vert_{\dot{B}^{\frac 1 2+\frac 2 p}_{2,\infty}}^p\Vert\partial_3v^3\Vert_{\mathcal{H}^{\theta,r}}^2+\Vert v^3\Vert_{\dot{B}^{\frac 1 2+\frac 2 p}_{2,\infty}}^2 \left\Vert\omega_{\frac r 2}\right\Vert_2
		^{4(\alpha(r)+\frac 1 p)}\left\Vert\nabla\omega_{\frac r 2}\right\Vert_2^{2-\frac 2 p}.
	\end{aligned}
\end{equation}
Thanks to \eqref{w estimate} and \eqref{v^3 estimate}, by virtue of the Gronwall argument we prove Theorem \ref{main} in the case $p>4$.

\section{Estimate of $(\omega,v^3)$: Case $2<p\leq4$}

\quad We first present two key lemmas that fundamentally underpin the subsequent estimation.

\begin{lemm}\label{embed4}
	For any $s>0$, we have
	$$
	\left\Vert\left(2^{ks}\Vert\Delta_k^{\rm h}f\Vert_p\right)_k\right\Vert_{l^q(\mathbb{Z})}\lesssim\Vert f\Vert_{\dot{B}^s_{p,q}},\ \left\Vert\left(2^{ls}\Vert\Delta_l^{\rm v}f\Vert_p\right)_l\right\Vert_{l^q(\mathbb{Z})}\lesssim\Vert f\Vert_{\dot{B}^s_{p,q}}.
	$$
\begin{proof}
	According to the supports of the Fourier multipliers $\Delta^{\rm h}_k$ and $\Delta_j$, there exists a $ N_0\in\mathbb{Z}$, such that
	$$
	\begin{aligned}
		\left\Vert\left(2^{ks}\Vert\Delta_k^{\rm h}f\Vert_p\right)_k\right\Vert_{l^q(\mathbb{Z})} &\lesssim
		\left\Vert\left(2^{ks}\sum_{k\leq j+N_0}\Vert\Delta_j\Delta_k^{\rm h}f\Vert_p\right)_k\right\Vert_{l^q(\mathbb{Z})}\\
		&\lesssim
		\left\Vert\left(\sum_{k\leq j+N_0}2^{(k-j)s}2^{js}\Vert\Delta_jf\Vert_p\right)_k\right\Vert_{l^q(\mathbb{Z})}\lesssim \Vert f\Vert_{\dot{B}^s_{p,q}}.
	\end{aligned}
	$$
	The last inequality was obtained by Young's inequality. Similarly, we can get the second inequality along the same line.
\end{proof}
\end{lemm}

\begin{lemm}\label{S operator}
	For any $s<0$, we have
	$$
	\left\Vert\left(2^{ks}\Vert S_k^{\rm h}f\Vert_p\right)_k\right\Vert_{l^q(\mathbb{Z})}\lesssim \left\Vert\left(2^{ks}\Vert \Delta_k^{\rm h}f\Vert_p\right)_k\right\Vert_{l^q(\mathbb{Z})},\
	\left\Vert\left(2^{ls}\Vert S_l^{\rm v}f\Vert_p\right)_l\right\Vert_{l^q(\mathbb{Z})}\lesssim
	\left\Vert\left(2^{ls}\Vert\Delta_l^{\rm v}f\Vert_p\right)_l\right\Vert_{l^q(\mathbb{Z})}.
	$$
\begin{proof}
According to Young's inequality, we have
		$$
	\begin{aligned}
		\left\Vert\left(2^{ks}\Vert S_k^{\rm h}f\Vert_p\right)_k\right\Vert_{l^q(\mathbb{Z})} &\lesssim
		\left\Vert\left(2^{ks}\sum_{k'\leq k-1}\Vert\Delta_k^{\rm h}f\Vert_p\right)_k\right\Vert_{l^q(\mathbb{Z})}\\
		&=
		\left\Vert\left(\sum_{k'\leq k-1}2^{(k-k')s}2^{k's}\Vert\Delta_k^{\rm h}f\Vert_p\right)_k\right\Vert_{l^q(\mathbb{Z})}\lesssim \left\Vert\left(2^{ks}\Vert \Delta_k^{\rm h}f\Vert_p\right)_k\right\Vert_{l^q(\mathbb{Z})}.
	\end{aligned}
	$$
The second inequality can be proved by the same token.
\end{proof}
\end{lemm}

\begin{lemm}\label{embed5}
	For any $r<2$, we have
	$$\Vert a\Vert_{(\dot{B}^0_{r,2})_{\rm h}(\dot{B}^0_{r,2})_{\rm v}}\lesssim\Vert a\Vert_r,\
	\left\Vert\left(
	\Vert \Delta_k^{\rm h}a \Vert_r
	\right)_k\right\Vert_{l^2}\lesssim\Vert a\Vert_r,\ {\rm and}\ \left\Vert\left(
	\Vert \Delta_l^{\rm v}a \Vert_r
	\right)_l\right\Vert_{l^2}\lesssim\Vert a\Vert_r.$$
\end{lemm}
\begin{proof}
	By Minkowski's inequality and Lemma \ref{embed3}, we obtain that
	$$
	\begin{aligned}
		\Vert a\Vert_{(\dot{B}^0_{r,2})_{\rm h}(\dot{B}^0_{r,2})_{\rm v}}
		\lesssim&
		\left\Vert\left(\left\Vert\left(\left\Vert
		\Delta_k^{\rm h}\Delta_l^{\rm v} a
		\right\Vert_r
		\right)_l
		\right\Vert_{l^2}
		\right)_k
		\right\Vert_{l^2}
		\lesssim
		\left\Vert\left(\left\Vert\left\Vert\left(\left\Vert
		\Delta_k^{\rm h}\Delta_l^{\rm v} a
		\right\Vert_{L^r_{\rm v}}
		\right)_l
		\right\Vert_{l^2}
		\right\Vert_{L^r_{\rm h}}
		\right)_k
		\right\Vert_{l^2}\\
		&\lesssim
		\left\Vert\left(\left\Vert\left\Vert
		\Delta_k^{\rm h}a
		\right\Vert_{L^r_{\rm v}}
		\right\Vert_{L^r_{\rm h}}
		\right)_k
		\right\Vert_{l^2}
		\lesssim
		\left\Vert\left(\left\Vert\left\Vert
		\Delta_k^{\rm h}a
		\right\Vert_{L^r_{\rm h}}
		\right)_k
		\right\Vert_{l^2}
		\right\Vert_{L^r_{\rm v}}\lesssim\Vert a\Vert_{L^r_{\rm v}L^r_{\rm h}}\\
		&\lesssim\Vert a\Vert_r.
	\end{aligned}
	$$
	Similarly, we can prove that the remaining two inequalities also hold.
\end{proof}

\subsection{Estimate of $\omega$}
\quad According to \eqref{curl L^r} and \eqref{I21 and I22}, we thus consider the estimate of $I_1$, $I_{21}$ and $I_{22}$. Recall that
$$I_1=\int\partial_3v^3|\omega|^r{\rm d}x,\ I_{21}:=\int\partial_hv^3|\nabla_{\rm h}|^{-1}\partial_3^2v^3\cdot\omega_{r-1}{\rm d}x\ {\rm and}\ I_{22}:=\int\partial_hv^3|\nabla_{\rm h}|^{-1}\partial_3\omega\cdot\omega_{r-1}{\rm d}x.
$$
Obviously \eqref{I1} is still remained for $2<p\leq4$. According to Lemma \ref{anisotropic product}, utilizing properties of the dual space we have for $\varepsilon>0$ small enough and  $r$ sufficiently close to $2^-$ such that $\varepsilon+\frac 1 {r'}-\frac 1 2>0$, then
\begin{equation}\label{I21 p<4}
	\begin{aligned}
		|I_{21}|&\lesssim\Vert v^3\Vert_{(\dot{B}^{\frac 2 p+\varepsilon}_{2,\infty})_{\rm h} (\dot{B}^{\frac 1 2-\varepsilon}_{2,2})_{\rm v}}
		\Vert |\nabla_{\rm h}|^{-1}\partial_3^2v^3\cdot\omega_{r-1}\Vert_{(\dot{B}^{-\frac 2 p-\varepsilon}_{2,1})_{\rm h} (\dot{B}^{\varepsilon-\frac 1 2}_{2,2})_{\rm v}}\\
		&\lesssim \Vert v^3\Vert_{\dot{B}^{\frac 1 2+\frac 2 p}_{2,\infty}}
		\Vert|\nabla_{\rm h}|^{-1}\partial^2 v^3\Vert_{(\dot{B}^{-\delta+1}_{2,2})_{\rm h} (\dot{B}^{0}_{2,2})_{\rm v}}
		\Vert\omega_{r-1}\Vert_{(\dot{B}^{\frac 1 2+\frac 1 r-\frac 2 p-\varepsilon}_{r',1})_{\rm h} (\dot{B}^{\varepsilon+\frac 1 {r'}-\frac 1 2}_{r',1})_{\rm v}}\\
		&\lesssim \Vert v^3\Vert_{\dot{B}^{\frac 1 2+\frac 2 p}_{2,\infty}}
		\Vert|\nabla_{\rm h}|^{-\delta}\partial^2 v^3\Vert_2
		\Vert\omega_{r-1}\Vert_{\dot{B}^{1-\frac 2 p}_{r',1}}\\
		&\lesssim \Vert v^3\Vert_{\dot{B}^{\frac 1 2+\frac 2 p}_{2,\infty}}\Vert\nabla\omega_{\frac r 2}\Vert_2^{1-\frac 2 p}\Vert\omega_{\frac r 2}\Vert_2^{\frac 2 {r'}-1+\frac 2 p}
		\Vert|\nabla_{\rm h}|^{-\delta}\partial^2 v^3\Vert_2.
	\end{aligned}
\end{equation}
The estimate of $I_{22}$ is the following proposition.
\begin{prop}
	For $2<p\leq4$, we have
	\begin{equation}\label{I22 p<4}
		\left|\int\partial_hv^3|\nabla_{\rm h}|^{-1}\partial_3\omega\cdot\omega_{r-1}{\rm d}x\right|
		\lesssim\Vert v^3\Vert_{\dot{B}^{\frac 1 2+\frac 2 p}_{2,\infty}}\Vert\nabla\omega_{\frac r 2}\Vert_2^{2-\frac 2 p}\Vert\omega_{\frac r 2}\Vert_2^{\frac 2 p}.
	\end{equation}
\end{prop}
\begin{proof}
	For $2<p<4$, since $\frac 2 p -\frac 1 2>0$, according to Lemma \ref{embed4} we have
	$$
	\begin{aligned}
		&\ \left|\int\partial_hv^3|\nabla_{\rm h}|^{-1}\partial_3\omega\cdot\omega_{r-1}{\rm d}x\right|\\
		&\lesssim
		\left(\sup_{l}2^{l(\frac 2 p-\frac 1 2)}\Vert\Delta_l^{\rm v}\partial_{\rm h}v^3\Vert_2\right)
		\left(\sum_{l}2^{l(\frac 1 2-\frac 2 p)}\Vert\Delta_l^{\rm v}(|\nabla_{\rm h}|^{-1}\partial\omega\cdot\omega_{r-1})\Vert_2\right)\\
		&\lesssim \Vert v^3\Vert_{\dot{B}^{\frac 1 2+\frac 2 p}_{2,\infty}}
		\bigg(\sum_{l}2^{l(\frac 1 2-\frac 2 p)}
		\big(
		\Vert\Delta_l^{\rm v}T^{v}(|\nabla_{\rm h}|^{-1}\partial\omega,\omega_{r-1})\Vert_2
		+\Vert\Delta_l^{\rm v}\tilde{T}^{v}(|\nabla_{\rm h}|^{-1}\partial\omega,\omega_{r-1})\Vert_2\\
		&\quad\quad\quad\quad\quad\quad\quad\quad\quad\quad\quad\quad\quad\quad+\Vert\Delta_l^{\rm v}R^{v}(|\nabla_{\rm h}|^{-1}\partial\omega,\omega_{r-1})\Vert_2
		\big)\bigg)\\
		&=:\Vert v^3\Vert_{\dot{B}^{\frac 1 2+\frac 2 p}_{2,\infty}}
		(K_1+K_2+K_3).
	\end{aligned}
	$$
	Therefore according to Lemma \ref{embed4}, Lemma \ref{S operator}, Lemma \ref{Bernstein} and the Sobolev inequality, we have
	$$
	\begin{aligned}
		K_1&\lesssim
		\sum_{\substack{l\in\mathbb{Z}\\|l'-l|\leq 4}}2^{l(\frac 1 2-\frac 2 p)}\Vert S_{l'-1}^{\rm v}|\nabla_{\rm h}|^{-1}\partial\omega\cdot\Delta_{l'}^{\rm v}\omega_{r-1}\Vert_2\\
		&\lesssim
		\sum_{\substack{l\in\mathbb{Z}\\|l'-l|\leq 4}}2^{l(\frac 1 2-\frac 2 p)}2^{\frac {l'} 2}\Vert S_{l'-1}^{\rm v}|\nabla_{\rm h}|^{-1}\partial\omega\Vert_{L^r_{\rm v}L^{(\frac 1 r-\frac 1 2)^{-1}}_{\rm h}}
		\Vert\Delta_{l'}^{\rm v}\omega_{r-1}\Vert_{r'}\\
		&\lesssim
		\sum_{\substack{l\in\mathbb{Z}\\|l'-l|\leq 4}}
		\Vert|\nabla_{\rm h}|^{-1}\partial\omega\Vert_{L^r_{\rm v}L^{(\frac 1 r-\frac 1 2)^{-1}}_{\rm h}} 2^{l(\frac 1 2-\frac 2 p)}2^{\frac {l'} 2} \Vert\Delta_{l'}^{\rm v}\omega_{r-1}\Vert_{r'}
		\lesssim\Vert\partial\omega\Vert_r
		\left(\sum_{l}2^{l(1-\frac 2 p)}\Vert\Delta_{l'}^{\rm v}\omega_{r-1}\Vert_{r'}\right)\\
		&\lesssim \Vert\partial\omega\Vert_r \Vert\omega_{r-1}\Vert_{\dot{B}^{1-\frac 2 p}_{r',1}},
	\end{aligned}
	$$
	$$
	\begin{aligned}
		K_2&\lesssim
		\sum_{\substack{l\in\mathbb{Z}\\|l'-l|\leq 4}}
		2^{l(\frac 1 2-\frac 2 p)}\Vert S_{l'-1}^{\rm v}\omega_{r-1}\Vert_{L^\infty_{\rm v}L^{r'}_{\rm h}}\cdot
		2^{l'(\frac 1 r-\frac 1 2)}
		\Vert\Delta_{l'}^{\rm v}|\nabla_{\rm h}|^{-1}\partial\omega\Vert_{L^r_{\rm v}L^{(\frac 1 r-\frac 1 2)^{-1}}_{\rm h}}\\
		&\lesssim\Vert\partial\omega\Vert_r
		\left(\sum_{l}2^{l(\frac 1 r-\frac 2 p)}\Vert S_{l-1}^{\rm v}\omega_{r-1}\Vert_{L^\infty_{\rm v}L^{r'}_{\rm h}}
		\right)
		\lesssim\Vert\partial\omega\Vert_r
		\left(\sum_{l}2^{l(\frac 1 r-\frac 2 p)}\Vert \Delta_l^{\rm v}\omega_{r-1}\Vert_{L^\infty_{\rm v}L^{r'}_{\rm h}}
		\right)\\
		&\lesssim
		\Vert\partial\omega\Vert_r
		\left(\sum_{l}2^{l(1-\frac 2 p)}\Vert \Delta_l^{\rm v}\omega_{r-1}\Vert_{r'}
		\right)
		\lesssim\Vert\partial\omega\Vert_r \Vert\omega_{r-1}\Vert_{\dot{B}^{1-\frac 2 p}_{r',1}},\\
		K_3&\lesssim \sum_{l}2^{l(1-\frac 2 p)}
		\Vert\Delta_l^{\rm v}R^{\rm v}(|\nabla_{\rm h}\partial\omega,\omega_{r-1})\Vert_{L^1_{\rm v}L^2_{\rm h}}\\
		&\lesssim \sum_{\substack{l\in\mathbb{Z}\\l'\geq l-3}}
		2^{l(1-\frac 2 p)}\Vert\tilde{\Delta}_{l'}^{\rm v}|\nabla_{\rm h}|^{-1}\partial\omega\Vert_{L^r_{\rm v}L^{(\frac 1 r-\frac 1 2)^{-1}}_{\rm h}}\Vert\Delta_{l'}^{\rm v}\omega_{r-1}\Vert_{r'}\\
		&\lesssim\Vert\partial\omega\Vert_r
		\left(
		\sum_{\substack{l\in\mathbb{Z}\\l'\geq l-3}} 2^{(l-l')(1-\frac 2 p)}2^{l'(1-\frac 2 p)}
		\Vert\Delta_{l'}^{\rm v}\omega_{r-1}\Vert_{r'}
		\right)
		\lesssim\Vert\partial\omega\Vert_r
		\left(
		\sum_{l}2^{l(1-\frac 2 p)}
		\Vert\Delta_{l}^{\rm v}\omega_{r-1}\Vert_{r'}
		\right)\\
		&\lesssim \Vert\partial\omega\Vert_r \Vert\omega_{r-1}\Vert_{\dot{B}^{1-\frac 2 p}_{r',1}}.
	\end{aligned}
	$$
	Combine above three inequalities and Lemma \ref{Sobolev type}, we obtain that when $2<p<4$
	$$
		\left|\int\partial_hv^3|\nabla_{\rm h}|^{-1}\partial_3\omega\cdot\omega_{r-1}{\rm d}x\right|
		\lesssim\Vert v^3\Vert_{\dot{B}^{\frac 1 2+\frac 2 p}_{2,\infty}}\Vert\nabla\omega_{\frac r 2}\Vert_2^{2-\frac 2 p}\Vert\omega_{\frac r 2}\Vert_2^{\frac 2 p}.
	$$
	For $p=4$, we have
	$$
	\begin{aligned}
		\left|\int\partial_hv^3|\nabla_{\rm h}|^{-1}\partial_3\omega\cdot\omega_{r-1}{\rm d}x\right|
		&\lesssim
		\Vert\partial_{\rm h}v^3\Vert_{(\dot{B}^{\varepsilon}_{2,\infty})_{\rm v}(\dot{B}^{-\varepsilon}_{2,1})_{\rm h}}
		\Vert|\nabla_{\rm h}|^{-1}\partial\omega\cdot\omega_{r-1}\Vert_{(\dot{B}^{-\varepsilon}_{2,1})_{\rm v}(\dot{B}^{\varepsilon}_{2,\infty})_{\rm h}}\\
		&\lesssim
		\Vert v^3\Vert_{(\dot{B}^{\varepsilon}_{2,\infty})_{\rm v}(\dot{B}^{1-\varepsilon}_{2,1})_{\rm h}}
		\Vert|\nabla_{\rm h}|^{-1}\partial\omega\cdot\omega_{r-1}\Vert_{(\dot{B}^{-\varepsilon}_{2,1})_{\rm v}(\dot{B}^{\varepsilon}_{2,\infty})_{\rm h}}\\
		&\lesssim
		\Vert v^3\Vert_{\dot{B}^1_{2,\infty}}
		\Vert|\nabla_{\rm h}|^{-1}\partial\omega\cdot\omega_{r-1}\Vert_{(\dot{B}^{-\varepsilon}_{2,1})_{\rm v}(\dot{B}^{\varepsilon}_{2,\infty})_{\rm h}}.
	\end{aligned}
	$$
	By applying the Bony decomposition to $|\nabla_{\rm h}|^{-1}\partial\omega\cdot\omega_{r-1}$ for both horizontal and vertical variables, we have
	$$|\nabla_{\rm h}|^{-1}\partial\omega\cdot\omega_{r-1}=(T^{\rm h}+\tilde{T}^{\rm h}+R^{\rm h})(T^{\rm v}+\tilde{T}^{\rm v}+R^{\rm v})(|\nabla_{\rm h}|^{-1}\partial\omega,\omega_{r-1}).$$
	Choosing $\varepsilon>0$ small enough and $r$ sufficiently close to $2^-$ such that $\frac 1 r-\frac 1 2-\varepsilon<0$, according to Lemma \ref{embed3}, Lemma \ref{embed5} and the support of the Fourier transform of each terms, we compute that
	$$
	\begin{aligned}
		&\Vert T^{\rm  h}(T^{\rm v}+\tilde{T}^{\rm v})(|\nabla_{\rm h}|^{-1}\partial\omega,\omega_{r-1})\Vert_{(\dot{B}^{-\varepsilon}_{2,1})_{\rm v}(\dot{B}^{\varepsilon}_{2,\infty})_{\rm h}}\\
		\lesssim&\sum_{l}\sup_{k}2^{-l\varepsilon}2^{k\varepsilon}\sum_{\substack{|k'-k|\leq 4\\ |l'-l|\leq 4}}
		2^{l'(\frac 1 r-\frac 1 2)}\Vert S_{k'-1}^{\rm h}(S_{l'-1}^{\rm v}+\Delta_{l'}^{\rm v}) |\nabla_{\rm h}|^{-1}\partial\omega\Vert_{L^r_{\rm v}L^{(\frac 1 r-\frac 1 2)^{-1}}_{\rm h}}
		\Vert \Delta_{k'}^{\rm h}(S_{l'-1}^{\rm v}+\Delta_{l'}^{\rm v})\omega_{r-1}\Vert_{L^\infty_{\rm v}L^{r'}_{\rm h}}\\
		\lesssim&
		\Vert\partial\omega\Vert_r
		\sum_{l}\sup_{k}2^{-l\varepsilon}2^{k\varepsilon}\sum_{\substack{|k'-k|\leq 4\\ |l'-l|\leq 4}}
		2^{l'(\frac 1 r-\frac 1 2)}\Vert \Delta_{k'}^{\rm h}(S_{l'-1}^{\rm v}+\Delta_{l'}^{\rm v})\omega_{r-1}\Vert_{L^\infty_{\rm v}L^{r'}_{\rm h}}\\
		\lesssim&
		\Vert\partial\omega\Vert_r
		\Vert\omega_{r-1}\Vert_{(\dot{B}^{\frac 1 r-\frac 1 2-\varepsilon}_{\infty,1})_{\rm v}(\dot{B}^{\varepsilon}_{r',\infty})_{\rm h}}
		\lesssim
		\Vert\partial\omega\Vert_r
		\Vert\omega_{r-1}\Vert_{(\dot{B}^{\frac 1 2-\varepsilon}_{r',1})_{\rm v}(\dot{B}^{\varepsilon}_{r',1})_{\rm h}}
		\lesssim
		\Vert\partial\omega\Vert_r\Vert\omega_{r-1}\Vert_{\dot{B}^{\frac 1 2}_{r',1}},
	\end{aligned}
	$$
	$$
	\begin{aligned}
		&\Vert T^{\rm  h}R^{\rm v}(|\nabla_{\rm h}|^{-1}\partial\omega,\omega_{r-1})\Vert_{(\dot{B}^{-\varepsilon}_{2,1})_{\rm v}(\dot{B}^{\varepsilon}_{2,\infty})_{\rm h}}\\
		\lesssim&
		\sum_{l}\sup_{k}2^{l(\frac 1 2-\varepsilon)}2^{k\varepsilon}\Vert\Delta_k^{\rm h}\Delta_l^{\rm v}T^{\rm  h}R^{\rm v}(|\nabla_{\rm h}|^{-1}\partial\omega,\omega_{r-1})\Vert_{L^1_{\rm v}L^2_{\rm h}}\\
		\lesssim&
		\sum_{l}\sup_{k}\sum_{\substack{|k'-k|\leq 4\\ l'\geq l-3}}2^{l(\frac 1 2-\varepsilon)}2^{k\varepsilon}
		\Vert S_{k'-1}^{\rm h}\tilde{\Delta}_{l'}^{\rm v}|\nabla_{\rm h}|^{-1}\partial\omega\Vert_{L^r_{\rm v}L^{(\frac 1 r-\frac 1 2)^{-1}}}
		\Vert\Delta_{k'}^{\rm h}\Delta_{l'}^{\rm v}\omega_{r-1}\Vert_{r'}\\
		\lesssim&
		\Vert\partial\omega\Vert_r
		\sum_{l}\sup_{k}\sum_{\substack{|k'-k|\leq 4\\ l'\geq l-3}}2^{(l-l')(\frac 1 2-\varepsilon)}2^{(k-k')\varepsilon}\cdot
		2^{l'(\frac 1 2-\varepsilon)}2^{k'\varepsilon}\Vert\Delta_{k'}^{\rm h}\Delta_{l'}^{\rm v}\omega_{r-1}\Vert_{r'}\\
		\lesssim&
		\Vert\partial\omega\Vert_r\Vert\omega_{r-1}\Vert_{(\dot{B}^{\frac 1 2-\varepsilon}_{r',1})_{\rm v}(\dot{B}^{\varepsilon}_{r',\infty})_{\rm h}}
		\lesssim
		\Vert\partial\omega\Vert_r\Vert\omega_{r-1}\Vert_{\dot{B}^{\frac 1 2}_{r',1}},\\
		&\Vert \tilde{T}^{\rm  h}(T^{\rm v}+\tilde{T}^{\rm v})(|\nabla_{\rm h}|^{-1}\partial\omega,\omega_{r-1})\Vert_{(\dot{B}^{-\varepsilon}_{2,1})_{\rm v}(\dot{B}^{\varepsilon}_{2,\infty})_{\rm h}}\\
		\lesssim&
		\sum_{l}\sup_{k}\sum_{\substack{|k'-k|\leq 4\\ |l'-l|\leq 4}}2^{-(l-l')\varepsilon}2^{(k-k')\varepsilon}
		2^{l'(\frac 1 2-\frac 1 r)} 2^{k'(1-\frac 2 r)}\Vert\Delta_{k'}^{\rm h}(S_{l'-1}^{\rm v}+\Delta_{l'}^{\rm v})|\nabla_{\rm h}|^{-1}\partial\omega\Vert_2\\
		&\quad\quad\quad\quad\quad\quad\quad\quad\quad\quad\quad\quad\quad\quad\quad\quad\quad
		\times 2^{l'(\frac 1 r-\frac 1 2-\varepsilon)}2^{k'(\varepsilon-\frac 2 {r'})}
		\Vert S_{k'-1}^{\rm h}(S_{l'-1}^{\rm v}+\Delta_{l'}^{\rm v})\omega_{r-1}\Vert_\infty\\
		\lesssim&
		\Vert|\nabla_{\rm h}|^{-1}\partial\omega\Vert_{(\dot{B}^{\frac 1 2-\frac 1 r}_{2,\infty})_{\rm v}(\dot{B}^{1-\frac 2 r}_{2,\infty})_{\rm h}}
		\Vert\omega_{r-1}\Vert_{(\dot{B}^{\frac 1 r-\frac 1 2-\varepsilon}_{\infty,1})_{\rm v}(\dot{B}^{\varepsilon-\frac 2 {r'}}_{\infty,\infty})_{\rm h}}
		\lesssim
		\Vert\partial\omega\Vert_r\Vert\omega_{r-1}\Vert_{\dot{B}^{\frac 1 2}_{r',1}},\\
		&\Vert \tilde{T}^{\rm  h}R^{\rm v}(|\nabla_{\rm h}|^{-1}\partial\omega,\omega_{r-1})\Vert_{(\dot{B}^{-\varepsilon}_{2,1})_{\rm v}(\dot{B}^{\varepsilon}_{2,\infty})_{\rm h}}\\
		\lesssim&
		\sum_{l}\sup_{k}2^{l(\frac 1 2-\varepsilon)}2^{k(\varepsilon+\frac 2 r-1)}\Vert\Delta_k^{\rm h}\Delta_l^{\rm v}\tilde{T}^{\rm  h}R^{\rm v}(|\nabla_{\rm h}|^{-1}\partial\omega,\omega_{r-1})\Vert_{L^r_{\rm h}L^1_{\rm v}}\\
		\lesssim&
		\sum_{l}\sup_{k}\sum_{\substack{|k'-k|\leq 4\\ l'\geq l-3}}2^{(l-l')(\frac 1 2-\varepsilon)}2^{(k-k')(\frac 2 r-1+\varepsilon)}
		\cdot 2^{k'}\Vert \tilde{\Delta}_{l'}^{\rm v}\Delta_{k'}^{\rm h}|\nabla_{\rm h}|^{-1}\partial\omega\Vert_r\\
		&\quad\quad\quad\quad\quad\quad\quad\quad\quad\quad\quad\quad\quad\quad\quad\quad\quad
		\times 2^{l'(\frac 1 2-\varepsilon)}2^{k'(\frac 2 r-2+\varepsilon)}\Vert S_{k'-1}^{\rm h}\Delta_{l'}^{\rm v}\omega_{r-1}\Vert_{L^\infty_{\rm h}L^{r'}_{\rm v}}\\
		\lesssim&
		\big(\sup_k 2^k\Vert \Delta_k^{\rm h}|\nabla_{\rm h}|^{-1}\partial\omega\Vert_r\big)
		\Vert\omega_{r-1}\Vert_{(\dot{B}^{\frac 1 2-\varepsilon}_{r',1})_{\rm v}(\dot{B}^{\frac 2 r-2+\varepsilon}_{\infty,\infty})_{\rm h}}
		\lesssim
		\Vert\partial\omega\Vert_r\Vert\omega_{r-1}\Vert_{\dot{B}^{\frac 1 2}_{r',1}},\\
		&\Vert R^{\rm h}(T^{\rm v}+R^{\rm v})(|\nabla_{\rm h}|^{-1}\partial\omega,\omega_{r-1})\Vert_{(\dot{B}^{-\varepsilon}_{2,1})_{\rm v}(\dot{B}^{\varepsilon}_{2,\infty})_{\rm h}}\\
		\lesssim&
		\sum_{l}\sup_{k}\sum_{\substack{k'\geq k-3\\ l'\geq l -4}}
		2^{l(\frac 1 2-\varepsilon)}2^{k\varepsilon}
		\Vert\tilde{\Delta}_{k'}^{\rm h}(S_{l'-1}^{\rm v}+\tilde{\Delta}_{l'}^{\rm v})|\nabla_{\rm h}|^{-1}\partial\omega\Vert_{L^r_{\rm v}L^{(\frac 1 r-\frac 1 2)^{-1}}_{\rm h}}\Vert\Delta_{k'}^{\rm h}\Delta_{l'}^{\rm v}\omega_{r-1}\Vert_{r'}\\
		\lesssim&
		\Vert\partial\omega\Vert_r
		\big(\sum_{l}\sup_{k}\sum_{\substack{k'\geq k-3\\ l'\geq l-4}}
		2^{(l-l')(\frac 1 2-\varepsilon)}2^{(k-k')\varepsilon}
		\cdot2^{l'(\frac 1 2-\varepsilon)}2^{k'\varepsilon}\Vert\Delta_{k'}^{\rm h}\Delta_{l'}^{\rm v}\omega_{r-1}\Vert_{r'}\big)
		\lesssim
		\Vert\partial\omega\Vert_r\Vert\omega_{r-1}\Vert_{(\dot{B}^{\frac 1 2-\varepsilon}_{r',1})_{\rm v}(\dot{B}^{\varepsilon}_{r',\infty})_{\rm h}}\\
		\lesssim&
		\Vert\partial\omega\Vert_r\Vert\omega_{r-1}\Vert_{\dot{B}^{\frac 1 2}_{r',1}},\\
		&\Vert R^{\rm h}\tilde{T}^{\rm v}(|\nabla_{\rm h}|^{-1}\partial\omega,\omega_{r-1})\Vert_{(\dot{B}^{-\varepsilon}_{2,1})_{\rm v}(\dot{B}^{\varepsilon}_{2,\infty})_{\rm h}}\\
		\lesssim&
		\sum_{l}\sup_{k}2^{l(\frac 1 r-\frac 1 2-\varepsilon)}2^{k\varepsilon}\sum_{\substack{k'\geq k-3\\ |l'-l|\leq 4}}
		\Vert \tilde{\Delta}_{k'}^{\rm h}\Delta_{l'}^{\rm v}|\nabla_{\rm h}|^{-1}\partial\omega\Vert_{L^r_{\rm v}L^{(\frac 1 r-\frac 1 2)^{-1}}_{\rm h}}
		\Vert\Delta_{k'}^{\rm h}S_{l'-1}\omega_{r-1}\Vert_{L^\infty_{\rm v}L^{r'}_{\rm h}}\\
		\lesssim&
		\Vert\partial\omega\Vert_r
		\sum_{l}\sup_{k}\sum_{\substack{k'\geq k-3\\ |l'-l|\leq 4}}2^{(l-l')(\frac 1 r-\frac 1 2-\varepsilon)}2^{(k-k')\varepsilon}\cdot
		2^{l'(\frac 1 r-\frac 1 2-\varepsilon)}2^{k'\varepsilon}\Vert\Delta_{k'}^{\rm h}S_{l'-1}\omega_{r-1}\Vert_{L^\infty_{\rm v}L^{r'}_{\rm h}}\\
		\lesssim&
		\Vert\partial\omega\Vert_r
		\Vert\omega_{r-1}\Vert_{(\dot{B}^{\frac 1 2-\varepsilon}_{r',1})_{\rm v}(\dot{B}^{\varepsilon}_{r',\infty})_{\rm h}}
		\lesssim
		\Vert\partial\omega\Vert_r\Vert\omega_{r-1}\Vert_{\dot{B}^{\frac 1 2}_{r',1}}.
	\end{aligned}
	$$
	Therefore combining above inequalities and Lemma \ref{Sobolev type}, we also obtain that for $p=4$
	$$
	\left|\int\partial_hv^3|\nabla_{\rm h}|^{-1}\partial_3\omega\cdot\omega_{r-1}{\rm d}x\right|
	\lesssim\Vert v^3\Vert_{\dot{B}^{1}_{2,\infty}}\Vert\nabla\omega_{\frac r 2}\Vert_2^{\frac 3 2}\Vert\omega_{\frac r 2}\Vert_2^{\frac 1 2}.
	$$
\end{proof}
Subsuming the estimate \eqref{I1}, \eqref{I21 p<4} and \eqref{I22 p<4} into \eqref{curl L^r}, we obtain the same estimate \eqref{w estimate} is valid for $2<p\leq4$.


\subsection{Estimate of $v^3$}
\quad Recall the equation of $v^3$ \eqref{v3}
$$
\partial_t\partial_k v^3+v\cdot\nabla\partial_k v^3-\Delta\partial_k v^3=-\partial_k v\cdot\nabla v^3+\partial_k\partial_3\Delta^{-1}\left(\sum_{l,m=1}^3\partial_m v^l\partial_l v^m\right).
$$
Taking the $L^2$ inner product of the $\partial_3v^3$ equation with $|\nabla_{\rm h}|^{-2\delta}\partial_kv^3$, we have
\begin{equation}\label{v3 L-2 norm}
	\begin{aligned}
		&\frac {\rm d}{{\rm d}t}\Vert|\nabla_{\rm h}|^{-\delta}\partial_kv^3\Vert_2^2+\Vert|\nabla_{\rm h}|^{-\delta}\nabla\partial_kv^3\Vert_2^2\\
		=&-\int(\partial_kv\cdot\nabla)v^3\cdot|\nabla_{\rm h}|^{-2\delta}\partial_k v^3{\rm d}x
		-\int (v\cdot \nabla)\partial_k v^3\cdot|\nabla_{\rm h}|^{-2\delta}\partial_k v^3{\rm d}x\\
		&-\int \partial_3\partial_k P\cdot|\nabla_{\rm h}|^{-2\delta}\partial_k v^3{\rm d}x
		=:J_1+J_2+J_3
	\end{aligned}
\end{equation}
with $P=\Delta^{-1}(\sum_{l,m=1}^3\partial_l v^m \partial_m v^l)$.

\noindent$\bullet$ Estimate of $J_1$.\\
$J_1$ can be written as
$$
J_1=-\int\partial_k v^3 \partial_3v^3\cdot|\nabla_{\rm h}|^{-2\delta}\partial_k v^3{\rm d}x-\int\partial_k v^{\rm h} \nabla_{\rm h}v^3\cdot|\nabla_{\rm h}|^{-2\delta}\partial_k v^3{\rm d}x=:J_{11}+J_{12}.
$$
According to Lemma \ref{embed2} and Lemma \ref{anisotropic product}, for $\varepsilon>0$ small enough we have
\begin{equation}\label{J11}
	\begin{aligned}
		|J_{11}|&\lesssim
		\Vert\partial_3v^3\Vert_{(\dot{B}^{\varepsilon}_{2,\infty})_{\rm h}(\dot{B}^{-\frac 1 2+\frac 2 p-\varepsilon}_{2,\infty})_{\rm v}}
		\Vert\partial_kv^3\cdot|\nabla_{\rm h}|^{-2\delta}\partial_k v^3\Vert_{(\dot{B}^{-\varepsilon}_{2,1})_{\rm h}(\dot{B}^{\frac 1 2-\frac 2 p+\varepsilon}_{2,1})_{\rm v}}\\
		&\lesssim\Vert v^3\Vert_{\dot{B}^{\frac 1 2+\frac 2 p}_{2,\infty}}
		\Vert\partial v^3\Vert_{(\dot{B}^{1-\delta-2\varepsilon}_{2,2})_{\rm h}(\dot{B}^{2\varepsilon}_{2,2})_{\rm v}}
		\Vert|\nabla_{\rm h}|^{-2\delta}\partial v^3\Vert\Vert_{(\dot{B}^{\delta+\varepsilon}_{2,2})_{\rm h}(\dot{B}^{1-\frac 2 p-\varepsilon}_{2,2})_{\rm v}}\\
		&\lesssim\Vert v^3\Vert_{\dot{B}^{\frac 1 2+\frac 2 p}_{2,\infty}}
		\Vert|\nabla_{\rm h}|^{-\delta}\partial v^3\Vert_{(\dot{B}^{1-2\varepsilon}_{2,2})_{\rm h}(\dot{B}^{2\varepsilon}_{2,2})_{\rm v}}
		\Vert|\nabla_{\rm h}|^{-\delta}\partial v^3\Vert\Vert_{(\dot{B}^{\varepsilon}_{2,2})_{\rm h}(\dot{B}^{1-\frac 2 p-\varepsilon}_{2,2})_{\rm v}}\\
		&\lesssim\Vert v^3\Vert_{\dot{B}^{\frac 1 2+\frac 2 p}_{2,\infty}}
		\Vert|\nabla_{\rm h}|^{-\delta}\partial v^3\Vert_2^{\frac 2 p}
		\Vert|\nabla_{\rm h}|^{-\delta}\partial^2 v^3\Vert_2^{2-\frac 2 p}.
	\end{aligned}
\end{equation}
Due to \eqref{decomposition}, $J_{12}$ can de reduced to
$$
J_{12}^{(1)}:=\int\partial_{\rm h}v^3\cdot|\nabla_{\rm h}|^{-1}\partial^2 v^3\cdot|\nabla_{\rm h}|^{-2\delta}\partial v^3{\rm d}x
\ \ {\rm and}\ \
J_{12}^{(2)}:=\int\partial_{\rm h}v^3\cdot|\nabla_{\rm h}|^{-1}\partial\omega\cdot|\nabla_{\rm h}|^{-2\delta}\partial v^3{\rm d}x.
$$
Also by Lemma \ref{anisotropic product}, we can deduce that for $\varepsilon>0$ small enough
\begin{equation}\label{J12^1}
	\begin{aligned}
		|J_{12}^{(1)}|
		&\lesssim \Vert|\nabla_{\rm h}|^{-1}\partial^2 v^3\Vert_{\dot{H}^{1-\delta,0}}
		\Vert \partial_{\rm h}v^3\cdot|\nabla_{\rm h}|^{-2\delta}\partial v^3\Vert_{\dot{H}^{\delta-1,0}}\\
		&\lesssim \Vert|\nabla_{\rm h}|^{-\delta}\partial^2v^3\Vert_2
		\Vert\partial_{\rm h}v^3\Vert_{(\dot{B}^{\frac 2 p+\varepsilon-1}_{2,\infty})_{\rm h}(\dot{B}^{\frac 1 2-\varepsilon}_{2,\infty})_{\rm v}}
		\Vert|\nabla_{\rm h}|^{-2\delta}\partial v^3\Vert_{(\dot{B}^{1-\frac 2 p+\delta-\varepsilon}_{2,2})_{\rm h}(\dot{B}^{\varepsilon}_{2,2})_{\rm v}}\\
		&\lesssim \Vert v^3\Vert_{\dot{B}^{\frac 1 2+\frac 2 p}_{2,\infty}}
		\Vert|\nabla_{\rm h}|^{-\delta}\partial^2v^3\Vert_2
		\Vert|\nabla_{\rm h}|^{-\delta}\partial v^3\Vert_{\dot{H}^{1-\frac 2 p}}\\
		&\lesssim \Vert v^3\Vert_{\dot{B}^{\frac 1 2+\frac 2 p}_{2,\infty}}
		\Vert|\nabla_{\rm h}|^{-\delta}\partial^2v^3\Vert_2^{2-\frac 2 p}
		\Vert|\nabla_{\rm h}|^{-\delta}\partial v^3\Vert_2^{\frac 2 p}.
	\end{aligned}
\end{equation}
And according to Lemma \ref{embed5} and Lemma \ref{anisotropic product}, for all $2<p\leq4$ we can choose $r$ sufficiently close to $2$ such that
\begin{equation}\label{J12^2}
	\begin{aligned}
		|J_{12}^{(2)}|&\lesssim
		\Vert |\nabla_{\rm h}|^{-1}\partial\omega\Vert_{(\dot{B}^{1}_{r,2})_{\rm h}(\dot{B}^{0}_{r,2})_{\rm v}}
		\Vert \partial_{\rm h} v^3|\nabla_{\rm h}|^{-2\delta}\partial v^3\Vert_{(\dot{B}^{-1}_{r',2})_{\rm h}(\dot{B}^{0}_{r',2})_{\rm v}}\\
		&\lesssim
		\Vert\partial\omega\Vert_{(\dot{B}^{0}_{r,2})_{\rm h}(\dot{B}^{0}_{r,2})_{\rm v}}
		\Vert \partial_{\rm h} v^3\Vert_{(\dot{B}^{-\delta}_{2,\infty})_{\rm h}(\dot{B}^{\frac 2 p-\frac 1 2+\delta}_{2,\infty})_{\rm v}}
		\Vert |\nabla_{\rm h}|^{-2\delta}\partial v^3\Vert_{(\dot{B}^{\delta+\frac 2 r-1}_{2,2})_{\rm h}(\dot{B}^{\frac 1 r+\frac 1 2-\frac 2 p-\delta}_{2,2})_{\rm v}}\\
		&\lesssim
		\Vert\partial\omega\Vert_r
		\Vert v^3\Vert_{(\dot{B}^{1-\delta}_{2,\infty})_{\rm h}(\dot{B}^{\frac 2 p-\frac 1 2+\delta}_{2,\infty})_{\rm v}}
		\Vert |\nabla_{\rm h}|^{-\delta}\partial v^3\Vert_{(\dot{B}^{\frac 2 r-1}_{2,2})_{\rm h}(\dot{B}^{\frac 1 r+\frac 1 2-\frac 2 p-\delta}_{2,2})_{\rm v}}\\
		&\lesssim
		\Vert\partial\omega\Vert_r
		\Vert v^3\Vert_{\dot{B}^{\frac 1 2+\frac 2 p}_{2,\infty}}
		\Vert|\nabla_{\rm h}|^{-\delta}\partial^2v^3\Vert_2^{1-\frac 2 p}
		\Vert|\nabla_{\rm h}|^{-\delta}\partial v^3\Vert_2^{\frac 2 p}.
	\end{aligned}
\end{equation}
According to \eqref{J11}, \eqref{J12^1} and \eqref{J12^2}, we get that
\begin{equation}\label{J1}
	|J_1|\lesssim\Vert v^3\Vert_{\dot{B}^{\frac 1 2+\frac 2 p}_{2,\infty}}
	\bigg(
	\Vert|\nabla_{\rm h}|^{-\delta}\partial^2v^3\Vert_2^{2-\frac 2 p}
	\Vert|\nabla_{\rm h}|^{-\delta}\partial v^3\Vert_2^{\frac 2 p}
	+\Vert|\nabla_{\rm h}|^{-\delta}\partial^2v^3\Vert_2^{1-\frac 2 p}
	\Vert|\nabla_{\rm h}|^{-\delta}\partial v^3\Vert_2^{\frac 2 p} \Vert\partial\omega\Vert_r
	\bigg).
\end{equation}

\noindent$\bullet$ Estimate of $J_2$.

Integrating by parts and combining the decomposition \eqref{decomposition}, $J_2$ can be divided into three parts:
$$J_{21}:=\int v^3\cdot\partial v^3|\nabla_h|^{-2\delta}\partial^2 v^3{\rm d}x,\
J_{22}:=\int |\nabla_{\rm h}|^{-1}\omega\cdot\partial v^3|\nabla_h|^{1-2\delta}\partial v^3{\rm d}x\ \ {\rm and}$$
$$J_{23}:=\int|\nabla_{\rm h}|^{-1}\partial v^3\cdot\partial v^3|\nabla_h|^{1-2\delta}\partial v^3{\rm d}x.$$
Similarly, by Lemma \ref{anisotropic product},
$$
\begin{aligned}
	|J_{21}|&\lesssim\Vert|\nabla_{\rm h}|^{-2\delta}\partial^2 v^3\Vert_{\dot{H}^{\delta,0}}\Vert v^3\cdot\partial v^3\Vert_{\dot{H}^{-\delta,0}}
	\lesssim\Vert|\nabla_{\rm h}|^{-\delta}\partial^2 v^3\Vert_2
	\Vert v^3\Vert_{(\dot{B}^{\frac 2 p+\varepsilon}_{2,\infty})_{\rm h}(\dot{B}^{\frac 1 2\varepsilon}_{2,\infty})_{\rm v}}
	\Vert\partial v^3\Vert_{\dot{H}^{1-\frac 2 p-\varepsilon-\delta,\varepsilon}}\\
	&\lesssim\Vert v^3\Vert_{\dot{B}^{\frac 1 2+\frac 2 p}_{2,\infty}}
	\lesssim\Vert|\nabla_{\rm h}|^{-\delta}\partial^2 v^3\Vert_2^{2-\frac 2 p}
	\lesssim\Vert|\nabla_{\rm h}|^{-\delta}\partial v^3\Vert_2^{\frac 2 p}.
\end{aligned}
$$
and for $s_1,s_2>0$ and $s_1+s_2\leq1$, according to Corollary 5.2 in \cite{Dong2019}, for $r<2$,
$$\Vert a\Vert_{(\dot{B}^{s_1}_{r,2})_{\rm h}(\dot{B}^{s_2}_{r,2})_{\rm v}}\lesssim \Vert a\Vert_{\dot{B}^{s_1+s_2}_{r,2}}\lesssim \Vert |\nabla|^{s_1+s_2}a\Vert_{\dot{B}^{0}_{r,2}}\lesssim\Vert |\nabla|^{s_1+s_2}a\Vert_r\lesssim\Vert a\Vert_r^{1-s_1-s_2}\Vert\nabla a\Vert_r^{s_1+s_2}.$$
Therefore
$$
\begin{aligned}
	|J_{22}|&\lesssim
	\Vert |\nabla_{\rm h}|^{-1}\omega\Vert_{(\dot{B}^{1+\varepsilon}_{r,2})_{\rm h}(\dot{B}^{1-\frac 2 p-\varepsilon}_{r,2})_{\rm v}}
	\Vert\partial v^3\cdot|\nabla_{\rm h}|^{1-2\delta}\partial v^3\Vert_{(\dot{B}^{-1-\varepsilon}_{r',2})_{\rm h}(\dot{B}^{-1+\frac 2 p+\varepsilon}_{r',2})_{\rm v}}\\
	&\lesssim
	\Vert \omega\Vert_{(\dot{B}^{\varepsilon}_{r,2})_{\rm h}(\dot{B}^{1-\frac 2 p-\varepsilon}_{r,2})_{\rm v}}
	\Vert \partial v^3\Vert_{(\dot{B}^{\frac 2 p-\frac 1 2-\varepsilon}_{2,\infty})_{\rm h}(\dot{B}^{\varepsilon}_{2,\infty})_{\rm v}}
	\Vert |\nabla_{\rm h}|^{1-2\delta}\partial v^3\Vert_{\dot{H}^{\frac 2 r-\frac 2 p-\frac 1 2,\frac 2 p-\frac 1 r-1}}\\
	&\lesssim \Vert v^3\Vert_{\dot{B}^{\frac 1 2+\frac 2 p}_{2,\infty}}
	\Vert\omega\Vert_r^{\frac 2 p}\Vert\nabla\omega\Vert_r^{1-\frac 2 p}
	\Vert|\nabla_{\rm h}|^{-\delta}\partial^2 v^3\Vert_2,\\
	|J_{23}|&\lesssim
	\Vert |\nabla_{\rm h}|^{-1}\partial v^3\Vert_{\dot{H}^{1-\delta,\frac 1 2-\varepsilon}}
	\Vert\partial v^3\cdot|\nabla_{\rm h}|^{1-2\delta}\partial v^3\Vert_{\dot{H}^{\delta-1,\varepsilon-\frac 1 2}}\\
	&\lesssim
	\Vert |\nabla_{\rm h}|^{-\delta}\partial v^3\Vert_2^{\frac 1 2+\varepsilon}
	\Vert |\nabla_{\rm h}|^{-\delta}\partial^2 v^3\Vert_2^{\frac 1 2-\varepsilon}
	\Vert\partial v^3\Vert_{(\dot{B}^{\frac 2 p-\frac 1 2-\varepsilon}_{2,\infty})_{\rm h}(\dot{B}^{\varepsilon}_{2,\infty})_{\rm v}}
	\Vert |\nabla_{\rm h}|^{1-2\delta}\partial v^3\Vert_{\dot{H}^{\delta+\frac 1 2+\varepsilon-\frac 2 p,0}}\\
	&\lesssim \Vert v^3\Vert_{\dot{B}^{\frac 1 2+\frac 2 p}_{2,\infty}}
	\Vert |\nabla_{\rm h}|^{-\delta}\partial v^3\Vert_2^{\frac 2 p}
	\Vert |\nabla_{\rm h}|^{-\delta}\partial^2 v^3\Vert_2^{2-\frac 2 p}.
\end{aligned}
$$

Consequently, we have
\begin{equation}\label{J2}
	|J_2|\lesssim\Vert v^3\Vert_{\dot{B}^{\frac 1 2+\frac 2 p}_{2,\infty}}\Vert|\nabla_{\rm h}|^{-\delta}\partial^2v^3\Vert_2
	\bigg(
	\Vert|\nabla_{\rm h}|^{-\delta}\partial^2v^3\Vert_2^{1-\frac 2 p}
	\Vert|\nabla_{\rm h}|^{-\delta}\partial v^3\Vert_2^{\frac 2 p}
	+\Vert\omega\Vert_r^{\frac 2 p}\Vert\nabla\omega\Vert_r^{1-\frac 2 p}
	\bigg).
\end{equation}

\noindent$\bullet$ Estimate of $J_3$.

Since $\partial_i\partial_j\Delta^{-1}$ and $\partial_l\partial_m\Delta_{\rm h}^{-1}$ are bounded Fourier multipliers for all $i,j=1,2,3$ and $l,m=1,2$. Therefore, according to \eqref{decomposition}, the estimate of $J_3$ can be divided into
$$
J_{31}:=\int \partial_{\rm h} v^3\partial_3 v^{\rm h} |\nabla_h|^{-2\delta}\partial v^3{\rm d}x\ \ {\rm and}\ \ J_{32}:=\int (\partial_3 v^3+\omega)\cdot(\partial_3 v^3+\omega)\cdot |\nabla_h|^{-2\delta}\partial v^3{\rm d}x.
$$
The estimate of $J_{31}$ is similar to $J_{32}$. Therefore we only consider $J_{32}$. By dual argument, we see that
$$
\begin{aligned}
	|J_{32}|&\lesssim\Vert |\nabla_{\rm h}|^{-2\delta}\partial v^3\Vert_{(\dot{B}^{\frac 2 p-\frac 1 2+2\delta-\varepsilon}_{2,\infty})_{\rm h}(\dot{B}^{\varepsilon}_{2,\infty})_{\rm h}}
	\Vert (\partial_3 v^3+\omega)^2\Vert_{(\dot{B}^{-\frac 2 p+\frac 1 2-2\delta+\varepsilon}_{2,1})_{\rm h}(\dot{B}^{-\varepsilon}_{2,1})_{\rm h}}\\
	&\lesssim
	\Vert v^3\Vert_{\dot{B}^{\frac 1 2+\frac 2 p}_{2,\infty}}
	\left(
	\Vert \omega\Vert_{\dot{H}^{\frac 3 4-\frac 1 p-\delta+\frac \varepsilon 2,\frac 1 4-\frac \varepsilon 2}}^2
	+\Vert \partial v^3\Vert_{\dot{H}^{\frac 3 4-\frac 1 p-\delta+\frac \varepsilon 2,\frac 1 4-\frac \varepsilon 2}}^2
	\right)\\
	&\lesssim\Vert v^3\Vert_{\dot{B}^{\frac 1 2+\frac 2 p}_{2,\infty}}
	\left(
	\Vert \omega\Vert_{(\dot{B}^{\frac 3 4-\frac 1 p-\frac{2\delta} 3+\frac \varepsilon 2}_{r,2})_{\rm h}(\dot{B}^{\frac 1 4+\frac{\delta} 3-\frac \varepsilon 2}_{r,2})_{\rm v}}^2
	+\Vert\nabla_{\rm h}|^{-\theta} \partial v^3\Vert_{\dot{H}^{1-\frac 1 p}}^2
	\right)\\
	&\lesssim\Vert v^3\Vert_{\dot{B}^{\frac 1 2+\frac 2 p}_{2,\infty}}
	\left(
	\Vert|\nabla_{\rm h}|^{-\delta}\partial^2v^3\Vert_2^{2-\frac 2 p}
	\Vert|\nabla_{\rm h}|^{-\delta}\partial v^3\Vert_2^{\frac 2 p}
	+\Vert\omega\Vert_r^{\frac 2 p}\Vert\nabla\omega\Vert_r^{2-\frac 2 p}
	\right).
\end{aligned}
$$
Thus, we conclude that
\begin{equation}\label{J3}
	\begin{aligned}
		|J_3|\lesssim\Vert v^3\Vert_{\dot{B}^{\frac 1 2+\frac 2 p}_{2,\infty}}
		\bigg(
		\Vert|\nabla_{\rm h}|^{-\delta}\partial^2v^3\Vert_2^{2-\frac 2 p}&
		\Vert|\nabla_{\rm h}|^{-\delta}\partial v^3\Vert_2^{\frac 2 p}
		+\Vert\omega\Vert_r^{\frac 2 p}\Vert\nabla\omega\Vert_r^{2-\frac 2 p}\\
		&+\Vert|\nabla_{\rm h}|^{-\delta}\partial^2v^3\Vert_2^{1-\frac 2 p}
		\Vert|\nabla_{\rm h}|^{-\delta}\partial v^3\Vert_2^{\frac 2 p} \Vert\partial\omega\Vert_r
		\bigg).
	\end{aligned}
\end{equation}
Subsuming \eqref{J1}, \eqref{J2} and \eqref{J3} into \eqref{v3 L-2 norm}, we finally conclude that
\begin{equation}\label{v^3 estimate p<4}
	\begin{aligned}
		&\frac{\rm d}{{\rm d}t}\Vert|\nabla_{\rm h}|^{-\delta}\partial v^3\Vert_2^2
		+\Vert|\nabla_{\rm h}|^{-\delta}\partial^2v^3\Vert_2^2\\
		\lesssim&\Vert v^3\Vert_{\dot{B}^{\frac 1 2+\frac 2 p}_{2,\infty}}
		\bigg(
		\Vert|\nabla_{\rm h}|^{-\delta}\partial^2v^3\Vert_2^{1-\frac 2 p}
		\Vert|\nabla_{\rm h}|^{-\delta}\partial v^3\Vert_2^{\frac 2 p} \Vert\partial\omega\Vert_r
		+\Vert\omega\Vert_r^{\frac 2 p}\Vert\nabla\omega\Vert_r^{1-\frac 2 p}
		\Vert|\nabla_{\rm h}|^{-\delta}\partial^2 v^3\Vert_2\\
		&\quad\quad\quad\quad\quad\quad\ +\Vert|\nabla_{\rm h}|^{-\delta}\partial^2v^3\Vert_2^{2-\frac 2 p}
		\Vert|\nabla_{\rm h}|^{-\delta}\partial v^3\Vert_2^{\frac 2 p}
		+\Vert\omega\Vert_r^{\frac 2 p}\Vert\nabla\omega\Vert_r^{2-\frac 2 p}\bigg).
	\end{aligned}
\end{equation}
Combining \eqref{w estimate} and \eqref{v^3 estimate p<4}, we thus prove Theorem \ref{main} in the case $2<p\leq4$.

\begin{rema}

	
	For $p=2$, since we only have the Sobolev embedding
	$$\Vert a\Vert_{L^\varepsilon(\mathbb{R})}\lesssim\Vert a\Vert_{\dot{H}^{\frac 1 2-\varepsilon}(\mathbb{R})},\  \Vert a\Vert_{L^\varepsilon(\mathbb{R}^2)}\lesssim\Vert a\Vert_{\dot{H}^{1-2\varepsilon}(\mathbb{R}^2)},$$
	The paraproduct decomposition index of the norm $\Vert a\cdot b\Vert_{(\dot{B}^1_{2,1})_{\rm h}(\dot{B}^{\frac 1 2}_{2,1})_{\rm v}}$ is a critical case, and the dual norm $\Vert v^3\Vert_{(\dot{B}^1_{2,\infty})_{\rm h}(\dot{B}^{\frac 1 2}_{2,\infty})_{\rm v}}$ cannot obtain by above Sobolev's embedding. The case $p=2$ is still open.
\end{rema}

\appendix
\section{Appendix}
\begin{prop}\cite{BCD}\label{isotropic Besov}
	
	There exists a couple of smooth function $(\chi,\varphi)$ valued in $[0,1]$, such that $\kappa$ is supported in the ball $B\triangleq\left\{\xi\in\mathbb{R}^d: \frac 3 4\leq|\xi|\leq\frac 8 3\right\}$. Moreover
	$$
	\forall\xi\in\mathbb{R}^d, \chi(\xi)+\sum_{j\geq0}\varphi(2^{-j}\xi)=1,$$
	\noindent and
	$$
	{\rm supp}\ \varphi(2^{-j}\cdot)\cap {\rm supp}\ \varphi(2^{-j^{\prime}\cdot})=\emptyset,\ if\  |j-j^{\prime}|\geq2,$$
	$$
	{\rm supp}\ \chi(\cdot)\cap {\rm supp}\ \varphi(2^{-j^{\prime}\cdot})=\emptyset,\ if\ j\geq1.$$
\end{prop}

~~Then for all $u\in\mathcal{S}^\prime$, we can define the anisotropic homogeneous dyadic blocks as follows. Let
$$
\Delta_k^{\rm h} a=\mathcal{F}^{-1}\left(\varphi(2^{-k}|\xi_h|\hat{a})\right),\quad \Delta_l^{\rm v} a=\mathcal{F}^{-1}\left(\varphi(2^{-l}|\xi_3|\hat{a})\right),$$
$$
S_k^{\rm h} a=\mathcal{F}^{-1}\left(\chi(2^{-k}|\xi_h|\hat{a})\right),\quad S_l^{\rm v} a=\mathcal{F}^{-1}\left(\chi(2^{-l}|\xi_3|\hat{a})\right),$$
$$
\Delta_j a=\mathcal{F}^{-1}\left(\varphi(2^{-j}|\xi|\hat{a})\right),\quad S_j a=\mathcal{F}^{-1}\left(\chi(2^{-j}|\xi|\hat{a})\right).$$

\begin{defi}\cite{Zhang2017}
	Let $(p,r)\in[1,\infty]^2$ and $s\in\mathbb{R}$. Let us consider $u\in S_h'(\mathbb{R}^3)$, which means that $u\in S'(\mathbb{R}^3)$ and satisfied $\lim_{j\rightarrow-\infty}\Vert S_ju\Vert_\infty=0$. We set
	$$\Vert u\Vert_{\dot{B}^s_{p,r}}\triangleq\left\Vert\left(2^{js}\Vert\Delta_ju\Vert_p\right)_j\right\Vert_{l^r(\mathbb{Z})}.$$Then for $s<\frac 3 p$(or $s=\frac 3 p$ if $r=1$), we define $\dot{B}^s_{p,r}(\mathbb{R}^3)\triangleq\left\{u\in S_h'(\mathbb{R}^3)|\Vert u\Vert_{\dot{B}^s_{p,r}}<\infty\right\}$.

Moreover, if $k$ is a positive integer and if $\frac 3 p+k\leq s<\frac 3 p+k+1)$(or $s=\frac 3 p+k+1$ if $r=1$), then we define $\dot{B}^s_{p,r}$ as the subset of the distributions $u\in S_h'(\mathbb{R}^3)$ such that $\partial^\beta u\in\dot{B}^{s-k}_{p,r}$ whenever $|\beta|=k$.
\end{defi}

Similar to Definition \ref{isotropic Besov}, we also define the homogeneous anisotropic Besov space.

\begin{defi}\cite{Zhang2017}\label{anisotropic Besov}
	Let us define the space $\left(\dot{B}^{s_1}_{p,q_1}\right)_{\rm h}\left(\dot{B}^{s_2}_{p,q_2}\right)_{\rm v}$ as the spce of distribution in $S_h'$ such that
	$$\Vert u\Vert_{\left(\dot{B}^{s_1}_{p,q_1}\right)_{\rm h}\left(\dot{B}^{s_2}_{p,q_2}\right)_{\rm v}}\triangleq\left(\sum_{k\in{\mathbb{Z}}}2^{q_1ks_1}\left(\sum_{l\in\mathbb{Z}}2^{q_2ls_2}\Vert\Delta_k^{\rm h}\Delta_l^{\rm v}u\Vert_p^{q_2}\right)^{ {q_1}/{q_2}}\right)^{1/{q_1}}$$
	is finite.
\end{defi}

We remark that when $p=q_1=q_2=2$, the anisotropic Besov space $\left(\dot{B}^{s_1}_{p,q_1}\right)_{\rm h}\left(\dot{B}^{s_2}_{p,q_2}\right)_{\rm v}$ coincides with the classical homogeneous anistropic Sobolev space $\dot{H}^{s_1,s_2}$. And we have the following anisotropic Bernstein type lemma:

\begin{lemm}\cite{Zhang2007,Zhang2017,Paicu2005}\label{Bernstein}
	Let $\mathcal{B}_{\rm h}$(resp. $\mathcal{B}_{\rm v}$) is a ball of $\mathbb{R}^2_{\rm h}$(resp. $\mathbb{R}_{\rm v}$), and $\mathcal{C}_{\rm h}$(resp. $\mathcal{C}_{\rm v}$) a ring of $\mathbb{R}^2_{\rm h}$(resp. $\mathbb{R}_{\rm v}$). Let $1\leq p_2\leq p_1\leq\infty$ and $1\leq q_2\leq q_1\leq\infty$. Then it holds that:
	
	\noindent$\bullet$ If ${\rm supp}\ \hat{a}\subseteq 2^k\mathcal{B}_{\rm h}$, then
	$$\Vert\partial_{x_{\rm h}}^\alpha a\Vert_{L^{p_1}_{\rm h}L^{q_1}_{\rm v}}\lesssim 2^{k(|\alpha|+2(1/p_2-1/p_1))}\Vert a\Vert_{L^{p_2}_{\rm h}L^{q_2}_{\rm v}}.$$
	
	\noindent$\bullet$ If ${\rm supp}\ \hat{a}\subseteq 2^l\mathcal{B}_{\rm v}$, then
	$$\Vert\partial_{x_3}^\beta a\Vert_{L^{p_1}_{\rm h}L^{q_1}_{\rm v}}\lesssim 2^{l(\beta+(1/p_2-1/p_1))}\Vert a\Vert_{L^{p_2}_{\rm h}L^{q_2}_{\rm v}}.$$
	
	\noindent$\bullet$ If ${\rm supp}\ \hat{a}\subseteq 2^k\mathcal{C}_{\rm h}$, then
	$$\Vert a\Vert_{L^{p_1}_{\rm h}L^{q_1}_{\rm v}}\lesssim 2^{-kN}\sup_{|\alpha|=N}\Vert\partial_{x_{\rm h}}^\alpha a\Vert_{L^{p_1}_{\rm h}L^{q_1}_{\rm v}}.$$
	
	\noindent$\bullet$ If ${\rm supp}\ \hat{a}\subseteq 2^k\mathcal{C}_{\rm v}$, then
	$$\Vert a\Vert_{L^{p_1}_{\rm h}L^{q_1}_{\rm v}}\lesssim 2^{-lN}\Vert\partial_{x_3}^N a\Vert_{L^{p_1}_{\rm h}L^{q_1}_{\rm v}}.$$
\end{lemm}

\begin{prop}\label{isotropic Bony}\cite{BCD}
	Let $a,\ b\in S'(\mathbb{R}^3)$, we have
	$$
	ab=T(a,b)+R(a,b)+\tilde{T}(a,b),
	$$
	here we have
	$$
	T(a,b)=\sum_{j\in\mathbb{Z}}S_{j-1}a\Delta_jb,\ R(a,b)=\sum_{j\in\mathbb{Z}}\Delta_ja\tilde{\Delta}_jb, \ {\rm and}\ \tilde{T}(a,b)=T(b,a)
	$$
	with $\tilde{\Delta}_jb=\sum_{l=j-1}^{j+1}\Delta_l b$.
\end{prop}

\smallskip
	\noindent\textbf{Acknowledgments}
This work was partially supported by the National Key R\&D Program of China(No. 2021YFA1002100) and the National Natural Science Foundation of China (No.12171493 and No. 12471233).
	
	\smallskip
	\noindent\textbf{Conflict of interest}
	
	The authors have no conflicts to disclose.
	
	\smallskip
	\noindent\textbf{Data Availability}
	
	The data that support the findings of this study are available from the corresponding author upon reasonable request.

	\phantomsection
	\addcontentsline{toc}{section}{\refname}
	\bibliographystyle{abbrv} 
	\bibliography{NSref}

@book {BCD,
    AUTHOR = {Bahouri, Hajer and Chemin, Jean-Yves and Danchin, Rapha\"{e}l},
     TITLE = {Fourier analysis and nonlinear partial differential equations},
    SERIES = {Grundlehren der mathematischen Wissenschaften [Fundamental
              Principles of Mathematical Sciences]},
    VOLUME = {343},
 PUBLISHER = {Springer, Heidelberg},
      YEAR = {2011},
     PAGES = {xvi+523},
      ISBN = {978-3-642-16829-1},
   MRCLASS = {35-02 (35L72 35Q30 42-02 42B37 76B03 76D03 76N10)},
  MRNUMBER = {2768550},
MRREVIEWER = {Peter\ R.\ Massopust},
       DOI = {10.1007/978-3-642-16830-7},
       URL = {https://doi.org/10.1007/978-3-642-16830-7},
}

@article {Wolf2021,
    AUTHOR = {Chae, D. and Wolf, J.},
     TITLE = {On the {S}errin-type condition on one velocity component for
              the {N}avier-{S}tokes equations},
   JOURNAL = {Arch. Ration. Mech. Anal.},
  FJOURNAL = {Archive for Rational Mechanics and Analysis},
    VOLUME = {240},
      YEAR = {2021},
    NUMBER = {3},
     PAGES = {1323--1347},
      ISSN = {0003-9527,1432-0673},
   MRCLASS = {35Q35 (76D05)},
  MRNUMBER = {4264947},
       DOI = {10.1007/s00205-021-01636-5},
       URL = {https://doi.org/10.1007/s00205-021-01636-5},
}

@article {Wu2024,
    AUTHOR = {Wang, Wendong and Wu, Di and Zhang, Zhifei},
     TITLE = {Scaling-invariant {S}errin criterion via one velocity
              component for the {N}avier-{S}tokes equations},
   JOURNAL = {Ann. Inst. H. Poincar\'e{} C Anal. Non Lin\'eaire},
  FJOURNAL = {Annales de l'Institut Henri Poincar\'e{} C. Analyse Non
              Lin\'eaire},
    VOLUME = {41},
      YEAR = {2024},
    NUMBER = {1},
     PAGES = {159--185},
      ISSN = {0294-1449,1873-1430},
   MRCLASS = {35Q30 (35Q35 76D03 76D05)},
  MRNUMBER = {4706030},
       DOI = {10.4171/aihpc/77},
       URL = {https://doi.org/10.4171/aihpc/77},
}

@article {Miller2020,
    AUTHOR = {Miller, Evan},
     TITLE = {A regularity criterion for the {N}avier-{S}tokes equation
              involving only the middle eigenvalue of the strain tensor},
   JOURNAL = {Arch. Ration. Mech. Anal.},
  FJOURNAL = {Archive for Rational Mechanics and Analysis},
    VOLUME = {235},
      YEAR = {2020},
    NUMBER = {1},
     PAGES = {99--139},
      ISSN = {0003-9527,1432-0673},
   MRCLASS = {35Q30 (35B65 76D03 76D05)},
  MRNUMBER = {4062474},
       DOI = {10.1007/s00205-019-01419-z},
       URL = {https://doi.org/10.1007/s00205-019-01419-z},
}

@article {Dong2019,
    AUTHOR = {Li, Dong},
     TITLE = {On {K}ato-{P}once and fractional {L}eibniz},
   JOURNAL = {Rev. Mat. Iberoam.},
  FJOURNAL = {Revista Matem\'atica Iberoamericana},
    VOLUME = {35},
      YEAR = {2019},
    NUMBER = {1},
     PAGES = {23--100},
      ISSN = {0213-2230,2235-0616},
   MRCLASS = {35R11 (35Q35 35Q86)},
  MRNUMBER = {3914540},
MRREVIEWER = {Vincenzo\ Ambrosio},
       DOI = {10.4171/rmi/1049},
       URL = {https://doi.org/10.4171/rmi/1049},
}

@article {Cao2008,
    AUTHOR = {Cao, Chongsheng and Titi, Edriss S.},
     TITLE = {Regularity criteria for the three-dimensional
              {N}avier-{S}tokes equations},
   JOURNAL = {Indiana Univ. Math. J.},
  FJOURNAL = {Indiana University Mathematics Journal},
    VOLUME = {57},
      YEAR = {2008},
    NUMBER = {6},
     PAGES = {2643--2661},
      ISSN = {0022-2518,1943-5258},
   MRCLASS = {35Q30 (35B65 76D05)},
  MRNUMBER = {2482994},
MRREVIEWER = {Luigi\ Carlo\ Berselli},
       DOI = {10.1512/iumj.2008.57.3719},
       URL = {https://doi.org/10.1512/iumj.2008.57.3719},
}

@article {Zhou2010,
    AUTHOR = {Zhou, Yong and Pokorn\'y, Milan},
     TITLE = {On the regularity of the solutions of the {N}avier-{S}tokes
              equations via one velocity component},
   JOURNAL = {Nonlinearity},
  FJOURNAL = {Nonlinearity},
    VOLUME = {23},
      YEAR = {2010},
    NUMBER = {5},
     PAGES = {1097--1107},
      ISSN = {0951-7715,1361-6544},
   MRCLASS = {35Q30 (35B65 76D03 76D05)},
  MRNUMBER = {2630092},
       DOI = {10.1088/0951-7715/23/5/004},
       URL = {https://doi.org/10.1088/0951-7715/23/5/004},
}

@article {Zhang2007,
    AUTHOR = {Chemin, Jean-Yves and Zhang, Ping},
     TITLE = {On the global wellposedness to the 3-{D} incompressible
              anisotropic {N}avier-{S}tokes equations},
   JOURNAL = {Comm. Math. Phys.},
  FJOURNAL = {Communications in Mathematical Physics},
    VOLUME = {272},
      YEAR = {2007},
    NUMBER = {2},
     PAGES = {529--566},
      ISSN = {0010-3616,1432-0916},
   MRCLASS = {35Q30 (35B30 35D05 76D03 76D05)},
  MRNUMBER = {2300252},
MRREVIEWER = {Luigi\ Carlo\ Berselli},
       DOI = {10.1007/s00220-007-0236-0},
       URL = {https://doi.org/10.1007/s00220-007-0236-0},
}

@article {Paicu2005,
    AUTHOR = {Paicu, Marius},
     TITLE = {\'{{E}}quation anisotrope de {N}avier-{S}tokes dans des espaces
              critiques},
   JOURNAL = {Rev. Mat. Iberoamericana},
  FJOURNAL = {Revista Matem\'atica Iberoamericana},
    VOLUME = {21},
      YEAR = {2005},
    NUMBER = {1},
     PAGES = {179--235},
      ISSN = {0213-2230},
   MRCLASS = {35Q30 (35Q35 76D03 76D05 76D09)},
  MRNUMBER = {2155019},
MRREVIEWER = {Gerhard\ O.\ Str\"ohmer},
       DOI = {10.4171/RMI/420},
       URL = {https://doi.org/10.4171/RMI/420},
}

@article {Iskauriaza2003,
    AUTHOR = {Escauriaza, L. and Ser\"egin, G. A. and {\v{S}}ver{\'a}k, V.},
     TITLE = {{$L_{3,\infty}$}-solutions of {N}avier-{S}tokes equations and
              backward uniqueness},
   JOURNAL = {Uspekhi Mat. Nauk},
  FJOURNAL = {Uspekhi Matematicheskikh Nauk},
    VOLUME = {58},
      YEAR = {2003},
    NUMBER = {2(350)},
     PAGES = {3--44},
      ISSN = {0042-1316,2305-2872},
   MRCLASS = {35Q30 (76D03 76D05)},
  MRNUMBER = {1992563},
MRREVIEWER = {Grzegorz\ Karch},
       DOI = {10.1070/RM2003v058n02ABEH000609},
       URL = {https://doi.org/10.1070/RM2003v058n02ABEH000609},
}

@article {Hopf1951,
    AUTHOR = {Hopf, Eberhard},
     TITLE = {\"Uber die {A}nfangswertaufgabe f\"ur die hydrodynamischen
              {G}rundgleichungen},
   JOURNAL = {Math. Nachr.},
  FJOURNAL = {Mathematische Nachrichten},
    VOLUME = {4},
      YEAR = {1951},
     PAGES = {213--231},
      ISSN = {0025-584X,1522-2616},
   MRCLASS = {76.1X},
  MRNUMBER = {50423},
MRREVIEWER = {J.\ Kamp\'e{} de F\'eriet},
       DOI = {10.1002/mana.3210040121},
       URL = {https://doi.org/10.1002/mana.3210040121},
}

@article {Lady1967,
    AUTHOR = {Lad\v{y}zenskaja, O. A.},
     TITLE = {Uniqueness and smoothness of generalized solutions of
              {N}avier-{S}tokes equations},
   JOURNAL = {Zap. Nau\v cn. Sem. Leningrad. Otdel. Mat. Inst. Steklov.
              (LOMI)},
  FJOURNAL = {Zapiski Nau\v cnyh Seminarov Leningradskogo Otdelenija
              Matemati\v ceskogo Instituta im. V. A. Steklova Akademii Nauk
              SSSR (LOMI)},
    VOLUME = {5},
      YEAR = {1967},
     PAGES = {169--185},
   MRCLASS = {35.79 (76.00)},
  MRNUMBER = {236541},
MRREVIEWER = {Mikhail\ Borsuk},
}

@article {Leray1934,
    AUTHOR = {Leray, Jean},
     TITLE = {Sur le mouvement d'un liquide visqueux emplissant l'espace},
   JOURNAL = {Acta Math.},
  FJOURNAL = {Acta Mathematica},
    VOLUME = {63},
      YEAR = {1934},
    NUMBER = {1},
     PAGES = {193--248},
      ISSN = {0001-5962,1871-2509},
   MRCLASS = {99-04},
  MRNUMBER = {1555394},
       DOI = {10.1007/BF02547354},
       URL = {https://doi.org/10.1007/BF02547354},
}

@article {Prodi1959,
    AUTHOR = {Prodi, Giovanni},
     TITLE = {Un teorema di unicit\`a{} per le equazioni di
              {N}avier-{S}tokes},
   JOURNAL = {Ann. Mat. Pura Appl. (4)},
  FJOURNAL = {Annali di Matematica Pura ed Applicata. Serie Quarta},
    VOLUME = {48},
      YEAR = {1959},
     PAGES = {173--182},
      ISSN = {0003-4622},
   MRCLASS = {35.79},
  MRNUMBER = {126088},
MRREVIEWER = {J.\ L.\ Lions},
       DOI = {10.1007/BF02410664},
       URL = {https://doi.org/10.1007/BF02410664},
}

@incollection {Serrin1963,
    AUTHOR = {Serrin, James},
     TITLE = {The initial value problem for the {N}avier-{S}tokes equations},
 BOOKTITLE = {Nonlinear {P}roblems ({P}roc. {S}ympos., {M}adison, {W}is.,
              1962)},
     PAGES = {69--98},
 PUBLISHER = {Univ. Wisconsin Press, Madison, WI},
      YEAR = {1963},
   MRCLASS = {35.79},
  MRNUMBER = {150444},
MRREVIEWER = {G.\ Prodi},
}

@article {Zhang2017,
    AUTHOR = {Chemin, Jean-Yves and Zhang, Ping and Zhang, Zhifei},
     TITLE = {On the critical one component regularity for 3-{D}
              {N}avier-{S}tokes system: general case},
   JOURNAL = {Arch. Ration. Mech. Anal.},
  FJOURNAL = {Archive for Rational Mechanics and Analysis},
    VOLUME = {224},
      YEAR = {2017},
    NUMBER = {3},
     PAGES = {871--905},
      ISSN = {0003-9527,1432-0673},
   MRCLASS = {35Q30 (35B65 76D03 76D05)},
  MRNUMBER = {3621812},
MRREVIEWER = {Luigi\ Carlo\ Berselli},
       DOI = {10.1007/s00205-017-1089-0},
       URL = {https://doi.org/10.1007/s00205-017-1089-0},
}

@article {Lei2019,
    AUTHOR = {Han, Bin and Lei, Zhen and Li, Dong and Zhao, Na},
     TITLE = {Sharp one component regularity for {N}avier-{S}tokes},
   JOURNAL = {Arch. Ration. Mech. Anal.},
  FJOURNAL = {Archive for Rational Mechanics and Analysis},
    VOLUME = {231},
      YEAR = {2019},
    NUMBER = {2},
     PAGES = {939--970},
      ISSN = {0003-9527,1432-0673},
   MRCLASS = {35Q30 (35B65 76D05)},
  MRNUMBER = {3900817},
MRREVIEWER = {Luigi\ Carlo\ Berselli},
       DOI = {10.1007/s00205-018-1292-7},
       URL = {https://doi.org/10.1007/s00205-018-1292-7},
}

@article {Chemin2016,
    AUTHOR = {Chemin, Jean-Yves and Zhang, Ping},
     TITLE = {On the critical one component regularity for 3-{D}
              {N}avier-{S}tokes systems},
   JOURNAL = {Ann. Sci. \'Ec. Norm. Sup\'er. (4)},
  FJOURNAL = {Annales Scientifiques de l'\'Ecole Normale Sup\'erieure.
              Quatri\`eme S\'erie},
    VOLUME = {49},
      YEAR = {2016},
    NUMBER = {1},
     PAGES = {131--167},
      ISSN = {0012-9593,1873-2151},
   MRCLASS = {35Q30 (35B44 35B65)},
  MRNUMBER = {3465978},
MRREVIEWER = {Lorenzo\ Brandolese},
       DOI = {10.24033/asens.2278},
       URL = {https://doi.org/10.24033/asens.2278},
}

@article {Kozono2000,
    AUTHOR = {Kozono, Hideo and Taniuchi, Yasushi},
     TITLE = {Bilinear estimates in {BMO} and the {N}avier-{S}tokes
              equations},
   JOURNAL = {Math. Z.},
  FJOURNAL = {Mathematische Zeitschrift},
    VOLUME = {235},
      YEAR = {2000},
    NUMBER = {1},
     PAGES = {173--194},
      ISSN = {0025-5874,1432-1823},
   MRCLASS = {76D03 (35Q30 76D05)},
  MRNUMBER = {1785078},
MRREVIEWER = {Piotr\ Bogus\l aw\ Mucha},
       DOI = {10.1007/s002090000130},
       URL = {https://doi.org/10.1007/s002090000130},
}

@article {Chen2006,
    AUTHOR = {Chen, Qionglei and Zhang, Zhifei},
     TITLE = {Space-time estimates in the {B}esov spaces and the
              {N}avier-{S}tokes equations},
   JOURNAL = {Methods Appl. Anal.},
  FJOURNAL = {Methods and Applications of Analysis},
    VOLUME = {13},
      YEAR = {2006},
    NUMBER = {1},
     PAGES = {107--122},
      ISSN = {1073-2772,1945-0001},
   MRCLASS = {35Q30 (35B45 35D10 76D03 76D05)},
  MRNUMBER = {2275874},
MRREVIEWER = {Sergey\ Nikolaevich\ Alekseenko},
       DOI = {10.4310/MAA.2006.v13.n1.a6},
       URL = {https://doi.org/10.4310/MAA.2006.v13.n1.a6},
}

@incollection {1,
    AUTHOR = {Neustupa, Jir\'i{} and Novotn\'y, Anton\'in and Penel,
              Patrick},
     TITLE = {An interior regularity of a weak solution to the
              {N}avier-{S}tokes equations in dependence on one component of
              velocity},
 BOOKTITLE = {Topics in mathematical fluid mechanics},
    SERIES = {Quad. Mat.},
    VOLUME = {10},
     PAGES = {163--183},
 PUBLISHER = {Dept. Math., Seconda Univ. Napoli, Caserta},
      YEAR = {2002},
   MRCLASS = {35Q30 (35B65 76D03 76D05)},
  MRNUMBER = {2051774},
MRREVIEWER = {Alp\ O.\ Eden},
}

@article {2,
    AUTHOR = {He, Cheng},
     TITLE = {Regularity for solutions to the {N}avier-{S}tokes equations
              with one velocity component regular},
   JOURNAL = {Electron. J. Differential Equations},
  FJOURNAL = {Electronic Journal of Differential Equations},
      YEAR = {2002},
     PAGES = {No. 29, 13},
      ISSN = {1072-6691},
   MRCLASS = {35Q30 (35D10 76D03 76D05)},
  MRNUMBER = {1907705},
MRREVIEWER = {Ewa\ Zadrzy\'nska},
}

@article {3,
    AUTHOR = {Zhou, Yong},
     TITLE = {A new regularity criterion for the {N}avier-{S}tokes equations
              in terms of the direction of vorticity},
   JOURNAL = {Monatsh. Math.},
  FJOURNAL = {Monatshefte f\"ur Mathematik},
    VOLUME = {144},
      YEAR = {2005},
    NUMBER = {3},
     PAGES = {251--257},
      ISSN = {0026-9255,1436-5081},
   MRCLASS = {35Q30 (35B65 76D03 76D05)},
  MRNUMBER = {2130277},
MRREVIEWER = {Luigi\ Carlo\ Berselli},
       DOI = {10.1007/s00605-004-0266-z},
       URL = {https://doi.org/10.1007/s00605-004-0266-z},
}

@article {4,
    AUTHOR = {Penel, Patrick and Pokorn\'y, Milan},
     TITLE = {Some new regularity criteria for the {N}avier-{S}tokes
              equations containing gradient of the velocity},
   JOURNAL = {Appl. Math.},
  FJOURNAL = {Applications of Mathematics},
    VOLUME = {49},
      YEAR = {2004},
    NUMBER = {5},
     PAGES = {483--493},
      ISSN = {0862-7940,1572-9109},
   MRCLASS = {35Q30 (35B65 76D03 76D05)},
  MRNUMBER = {2086090},
MRREVIEWER = {Pablo\ Braz e Silva},
       DOI = {10.1023/B:APOM.0000048124.64244.7e},
       URL = {https://doi.org/10.1023/B:APOM.0000048124.64244.7e},
}

@article {5,
    AUTHOR = {Kukavica, Igor and Ziane, Mohammed},
     TITLE = {Navier-{S}tokes equations with regularity in one direction},
   JOURNAL = {J. Math. Phys.},
  FJOURNAL = {Journal of Mathematical Physics},
    VOLUME = {48},
      YEAR = {2007},
    NUMBER = {6},
     PAGES = {065203, 10},
      ISSN = {0022-2488,1089-7658},
   MRCLASS = {35Q30 (76D03 76D05)},
  MRNUMBER = {2337002},
       DOI = {10.1063/1.2395919},
       URL = {https://doi.org/10.1063/1.2395919},
}

@article{Zhang2024,
  title={On the one time-varying component regularity criteria for 3-D Navier-Stokes equations},
  author={Liu, Yanlin and Zhang, Ping},
  journal={SIAM Journal on Mathematical Analysis},
  volume={56},
  number={5},
  pages={6213--6231},
  year={2024},
  publisher={SIAM}
}
\end{document}